\documentclass[onefignum,onetabnum]{siamart190516}

\usepackage{lipsum}
\usepackage{amsfonts}
\usepackage{graphicx}
\usepackage{epstopdf}
\usepackage{algorithmic}
\ifpdf
  \DeclareGraphicsExtensions{.eps,.pdf,.png,.jpg}
\else
  \DeclareGraphicsExtensions{.eps}
\fi

\newsiamremark{remark}{Remark}
\newsiamremark{hypothesis}{Hypothesis}
\crefname{hypothesis}{Hypothesis}{Hypotheses}
\newsiamthm{claim}{Claim}

\headers{IMEX methods for nonhydrostatic dynamics}{A. Steyer, C.J. Vogl, M. Taylor, and O. Guba}

\title{Efficient IMEX Runge-Kutta methods for nonhydrostatic dynamics\thanks{Support  for  this  work  was  provided  by  the  Department  of  Energy,  Office  of  Science  Scientific  Discovery  through Advanced Computing (SciDAC) project “A Non-hydrostatic Variable Resolution Atmospheric Model in ACME.”  \newline \text{  }\text{  }\text{  }\text{  } Sandia National Laboratories is a multi-mission laboratory managed and operated by National Technology and Engineering Solutions of Sandia, LLC., a wholly owned subsidiary of Honeywell International, Inc., for the U.S. Department of Energy’s National Nuclear Security Administration under contract DE-NA0003525.  This paper describes objective technical results and analysis. Any subjective views or opinions that might be expressed in the paper do not necessarily represent the views of the U.S. Department of Energy or the United States Government. \newline \text{  }\text{  }\text{  }\text{  }  This work was performed under the auspices of the U.S. Department of Energy by Lawrence Livermore National Laboratory under Contract DE-AC52-07NA27344.  LLNL-JRNL-777661.  \newline \text{  }\text{  }\text{  } This document was prepared as an account of work sponsored by an agency of the United States government. Neither the United States government nor Lawrence Livermore National Security, LLC, nor any of their employees makes any warranty, expressed or implied, or assumes any legal liability or responsibility for the accuracy, completeness, or usefulness of any information, apparatus, product, or process disclosed, or represents that its use would not infringe privately owned rights. Reference herein to any specific commercial product, process, or service by trade name, trademark, manufacturer, or otherwise does not necessarily constitute or imply its endorsement, recommendation, or favoring by the United States government or Lawrence Livermore National Security, LLC. The views and opinions of authors expressed herein do not necessarily state or reflect those of the United States government or Lawrence Livermore National Security, LLC, and shall not be used for advertising or product endorsement purposes.}}

\author{Andrew Steyer\thanks{Computational Science, Sandia National Laboratories, Albuquerque,
New Mexico, USA, (asteyer@sandia.gov).} \and  Christopher J. Vogl\thanks{Center for Applied Scientific Computing, Lawrence Livermore National Laboratory, Livermore, California, USA (vogl2@llnl.gov).} \and Mark Taylor\thanks{Computational Science, Sandia National Laboratories, Albuquerque,
New Mexico, USA, (mataylo@sandia.gov).} \and Oksana Guba\thanks{Computational Science, Sandia National Laboratories, Albuquerque,
New Mexico, USA,  (onguba@sandia.gov).}}

\usepackage{amsopn}

\ifpdf
\hypersetup{
  pdftitle={Efficient IMEX Runge-Kutta methods for nonhydrostatic dynamics},
  pdfauthor={A. Steyer, C.J. Vogl, M. Taylor, and O. Guba}
}
\fi
\usepackage{braket,amsfonts}

\usepackage{array}

\usepackage[caption=false]{subfig}
\usepackage{pgfplots}
\usepackage[nocompress]{cite}
\newsiamthm{examp}{Example}
\crefname{hypothesis}{Hypothesis}{Hypotheses}

\usepackage{algorithmic}

\usepackage{graphicx,epstopdf}
\Crefname{ALC@unique}{Line}{Lines}

\usepackage{amsopn}
\usepackage{amssymb}

\usepackage{xspace}
\usepackage{bold-extra}
\usepackage[most]{tcolorbox}

\colorlet{texcscolor}{blue!50!black}
\colorlet{texemcolor}{red!70!black}
\colorlet{texpreamble}{red!70!black}
\colorlet{codebackground}{black!25!white!25}

\lstdefinestyle{siamlatex}{%
  style=tcblatex,
  texcsstyle=*\color{texcscolor},
  texcsstyle=[2]\color{texemcolor},
  keywordstyle=[2]\color{texemcolor},
  moretexcs={cref,Cref,maketitle,mathcal,text,headers,email,url},
}

\tcbset{%
  colframe=black!75!white!75,
  coltitle=white,
  colback=codebackground, 
  colbacklower=white, 
  fonttitle=\bfseries,
  arc=0pt,outer arc=0pt,
  top=1pt,bottom=1pt,left=1mm,right=1mm,middle=1mm,boxsep=1mm,
  leftrule=0.3mm,rightrule=0.3mm,toprule=0.3mm,bottomrule=0.3mm,
  listing options={style=siamlatex}
}

\newtcblisting[use counter=example]{example}[2][]{%
  title={Example~\thetcbcounter: #2},#1}

\newtcbinputlisting[use counter=example]{\examplefile}[3][]{%
  title={Example~\thetcbcounter: #2},listing file={#3},#1}

\DeclareTotalTCBox{\code}{ v O{} }
{ 
  fontupper=\ttfamily\color{black},
  nobeforeafter,
  tcbox raise base,
  colback=codebackground,colframe=white,
  top=0pt,bottom=0pt,left=0mm,right=0mm,
  leftrule=0pt,rightrule=0pt,toprule=0mm,bottomrule=0mm,
  boxsep=0.5mm,
  #2}{#1}

\patchcmd\newpage{\vfil}{}{}{}
\newcommand{\pd}[3]{\frac{\partial^{#3} #1}{\partial #2^{^{#3}}}}

\newcommand{\bo}[1]{\mathbf{#1}}

\begin{document}
\maketitle

\begin{abstract}  
 We analyze the stability and accuracy (up to third order) of a new family of implicit-explicit Runge-Kutta (IMEX RK) methods.  This analysis expedites development of methods with various balances in the number of explicit stages and implicit solves.  We emphasize deriving methods with large stability regions for horizontally explicit vertically implicit (HEVI) partitionings of nonhydrostatic atmosphere models.  The IMKG2 and IMKG3 families of IMEX RK methods are formulated within this framework.  The HOMME-NH model with a HEVI partitioning is used for testing the accuracy and stability of various IMKG2-3 methods.  The efficiency of several IMKG2-3 methods is demonstrated in HOMME-NH and compared to other IMEX RK methods in the literature.

\end{abstract}




\section{Introduction}

Method-of-lines discretizations of time-dependent partial differential equations (PDEs) frequently result in stiff initial value problems (IVPs).  The dynamics of the resulting stiff IVPs may include processes evolving on separate time-scales.  Such IVPs can often be partitioned into a stiff term representing fast processes and a non-stiff term representing slow processes.  Many alternatives to traditional implicit methods exist for discretizing such partitioned IVPs, including implicit-explicit (IMEX), exponential, and multirate methods.  These alternatives can circumvent step-size restrictions intrinsic to traditional explicit methods, often at a lower computational cost than traditional implicit methods.  In this paper, we analyze a family of IMEX Runge-Kutta (RK) methods for stability and accuracy (up to third order).  Our focus is on deriving methods that are efficient for nonhydrostatic atmosphere models with a horizontally explicit vertically implicit (HEVI) partitioning. As such, we develop a HEVI partitioning of the HOMME-NH nonhydrostatic atmosphere model and compare the performance of IMEX RK methods, both derived herein and from the literature, for its integration in time.


After covering some preliminaries in Section \ref{sec:preliminaries}, we present the family of IMEX RK methods we analyze in Equation \eqref{eq:IMEXrk_ls} in Section \ref{sec:methods}.  The analysis includes methods with an arbitrarily large number of internal stages.  The number of implicit solves can be varied by setting diagonal entries in the implicit method's Butcher tableau to zero.  This enables deriving methods with efficient balances of explicit and implicit stages.  In Theorem \ref{thm:orderthm}, we derive simplified criteria for methods of the form \eqref{eq:IMEXrk_ls} to have second or third order accuracy.  The remaining free coefficients are chosen to optimize the explicit stability region on the imaginary axis, ensure that the implicit method is I-, A-, or L-stable (sufficient conditions for which are given in Theorems \ref{thm:Lstable} and  \ref{thm:Istable}), and improve the H-stability region (defined in Section \ref{sec:preliminaries}).  The IMKG2 and IMKG3 methods (Definition \ref{def:imexkg}) are subsequently introduced.   Although we emphasize deriving methods for atmosphere models with a HEVI partitioning, the analysis is general and can be used for other applications.  Double Butcher tableaux and various properties of the IMKG2-3 methods we derive are given in the appendix.  

In Sections \ref{sec:nh}-\ref{sec:experiments}, we derive a HEVI IMEX partitioning for the HOMME-NH nonhydrostatic atmosphere model that is then used to evaluate the performance of various methods.  The governing equations of HOMME-NH (Equation \eqref{eq:nhtheta}) support vertically propagating acoustic waves (Section \ref{sec:vertwav}) requiring stable numerical treatment.  The stiff terms generating these waves are isolated to the equations for vertical momentum and geopotential.  This results (Section \ref{sec:imexsplitting}) in a HEVI IMEX partitioning where the implicitly treated terms require the solution of relatively simple nonlinear equations that are independent of horizontal derivatives.  The nonlinear solvers can then be implemented without horizontal parallel communication.  The performance of various IMKG2-3 methods integrating HOMME-NH with this HEVI partitioning is investigated in Section \ref{sec:experiments}.  The most efficient of these methods can run with relatively large maximum stable step-sizes for a variety of vertical-to-horizontal aspect ratios.  These efficient IMKG2-3 methods have a faster time-to-solution than other IMEX RK methods we test from the literature. 

Our focus on IMEX methods is motivated by their frequent use in models of geophysical fluid flow \cite{GRL2010,GKC2013,GC2016,BD2012,LWW2013,LWW2014,UllrichJablonowski2012,ARKODE2018,Satoh2002}.  Order conditions for various partitioned and IMEX methods were derived in \cite{Hairer1981}.  Explicit formulas for the order conditions of the IMEX RK methods we analyze are given in \cite{JV2000}.  Understanding stability properties of IMEX methods is important for deriving efficient methods and has been extensively studied (see e.g. \cite{BHL1982,JHV1997,ARS1997} and more recently \cite{BSZ2016,JI2017,RSSZ2017}).  We exploit the technique, dating back at least to the early 1970s \cite{Houwen1972}, of increasing the maximum stable step-size by increasing the number of explicit stages.  The analysis pioneered in \cite{BD2012,LWW2013,LWW2014} is then used to improve the stability properties of our methods for integration of nonhydrostatic atmosphere models with a HEVI  partitioning \cite{Satoh2002}.  


The HOMME-NH nonhydrostatic atmosphere model used to evaluate the IMKG2-3 methods is based on the spectral element hydrostatic HOMME dynamic core \cite{camse2012,cam42013,FournierTaylor2010}.  HOMME-NH is expected to run at a variety of high (3km) and low (25-100km) horizontal resolutions.  It therefore requires time-integration methods that are efficient across a wide range of vertical-to-horizontal aspect ratios.  Semi-implicit and IMEX time-integration strategies have been employed in nonhydrostatic atmosphere models for many years (see e.g. \cite{tappwhite1976,Satoh2002}).  As mentioned above, the strategy we employ is the HEVI partitioning \cite{Satoh2002,BD2012,LWW2013,LWW2014,ARKODE2018,accerlatenuma} that treats stiff vertically propagating acoustic waves implicitly and everything else explicitly.

\section{Implicit-explicit Runge-Kutta methods}\label{sec:preliminaries}

\subsection{Formulation}

Consider an additively partitioned ODE
\begin{equation}\label{eq:ode}
\dot{x} = f(x,t) \equiv  n(x,t)+s(x,t), \quad f,n,s:\mathbb{R}^{d}\times \mathbb{R} \rightarrow \mathbb{R}^{d},
\end{equation}
where $d \in \mathbb{N}$ and $\dot{x}$ is the derivative of $x$ with respect to $t$. Given $r \in \mathbb{N}$ and real-valued arrays $b,\hat{b},c,\hat{c} \in \mathbb{R}^r$ and $A,\hat{A} \in \mathbb{R}^{r\times r}$ where $\hat{A}$ is lower triangular and $A$ is strictly lower triangular, we consider $r$-stage IMEX RK methods for approximating IVPs of \eqref{eq:ode} with initial condition $x(t_0) = x_0$ defined by  
\begin{equation}\label{eq:dirkIMEX}
\left\{\begin{array}{lcr}
x_{m+1} = x_m + \Delta t \sum_{k=1}^{r}( b_k n_{m,k} + \hat{b}_k s_{m,k})\\
g_{m,j} = E_{m,j} + \Delta t \hat{A}_{j,j} s_{m,k}, \quad j=1,\ldots,r \quad m \in \{0\} \cup \mathbb{N},
\end{array}\right. 
\end{equation}
where $\Delta t > 0$ is the step-size, $n_{m,k}:= n(g_{m,k},t_m + c_k \Delta t)$, $s_{m,k} := s(g_{m,k},t_m+ \hat c_k \Delta t)$, $t_m := t_0 + m\Delta t$, and
$$E_{m,j} :=\left\{  \begin{array}{cc} x_m & j = 1 
\\ x _{m} + \Delta t \sum_{k=1}^{j-1}( A_{j,k}n_{m,k} + \hat{A}_{j,k} s_{m,k} ) & j=2,\hdots,r.  \end{array} \right. $$
We represent \eqref{eq:dirkIMEX} with a double Butcher tableau:
\begin{equation}\label{eq:doubletable}
\def\arraystretch{1.2}
\begin{array}{c|c} c & A \\ \hline & b^T \end{array} \quad \begin{array}{c|c} \hat{c} & \hat{A} \\ \hline & \hat{b}^T \end{array}.
\end{equation}
The explicit RK method $\def\arraystretch{1.2}\begin{array}{c|c} c & A \\ \hline & b^T \end{array}$ is called the explicit method of \eqref{eq:dirkIMEX} and the implicit RK method $\def\arraystretch{1.2}\begin{array}{c|c} \hat{c} & \hat{A} \\ \hline & \hat{b}^T \end{array}$ is called the implicit method of \eqref{eq:dirkIMEX}.   If $\hat{A}$ has $\nu \leq r$ nonzero diagonal entries, then we say that \eqref{eq:dirkIMEX} has $\nu$ implicit stages.  If $b_j = A_{r,j}$ and $\hat{b}_j = \hat{A}_{r,j}$ for $j=1,\ldots,r$, then we say that \eqref{eq:dirkIMEX} is FSAL (first same as last).  If every nonzero diagonal entry of $\hat{A}$ is equal, then we say that \eqref{eq:dirkIMEX} is SD (single diagonal entry).







\subsection{Stability of explicit RK methods on the imaginary axis}



The stability theory of RK methods for hyperbolic PDEs is a well-established subject \cite{JN1981,Houwen1996}.  Important to our work is the following theorem bounding the intersection of the stability region of an explicit RK method with the imaginary axis.
\begin{theorem}
Given real numbers $a < b$ let $i \cdot [a,b] := \{ z \in \mathbb{C}: z = i \xi, \xi \in [a,b]\}$.  For an $r$-stage explicit RK method with $r \geq 2$ and stability region $\mathcal{S}$, the maximal interval $[a,b]$ such that $i \cdot [a,b] \in \mathcal{S}$ is contained in $[-r+1,r-1]$.
\end{theorem}
For a proof refer to \cite[Theorem 5.1]{JN1981},  \cite[Theorem 2]{Vichnevetsky1983}, or \cite[Chapter 4]{Houwen1977}.  The stability polynomials achieving the optimal stability limit ($i \cdot [-r+1,r-1] \subset \mathcal{S}$), referred to as the KGO  (Kinnmark and Grey optimal) polynomials, are given in \cite[Table 1]{KG1984a}.  We employ the third and fourth order accurate KGNO  (Kinnmark and Grey near optimal) polynomials \cite[Table 1]{KG1984b} when KGO polynomials do not attain the desired order of accuracy.  The stability region of explicit RK methods with KGNO stability polynomials contains $i \cdot [-r_0,r_0]$ where $r_0 = \sqrt{(r-1)^2 -1}$.

\subsection{Stability of IMEX methods with HEVI partitionings}\label{sec:imexstability}

 The following test equation (posed here in dimensionless form) has been proposed for characterizing the stability of IMEX methods for atmospheric models with a HEVI splitting \cite{BD2012,LWW2013,LWW2014}:
\begin{equation}\label{eq:hevitesteq}
\dot{u} = -i k_x  \mathcal{N} u -i k_z  \mathcal{S} u , \quad \mathcal{N} = \left[\begin{array}{ccc}0 & 0 & 1 \\ 0 & 0 & 0 \\ 1 & 0 & 0 \end{array}\right], \quad  \mathcal{S}=\left[\begin{array}{ccc}0 & 0 & 0 \\ 0 & 0 & 1 \\ 0 & 1 & 0 \end{array}\right].
\end{equation}
Here $k_x, k_z \in \mathbb{R}$ represent wave numbers of horizontally and vertically propagating waves and are referred to as horizontal and vertical wave numbers.  Let $K_x$ and $K_z$ denote the set of all horizontal and vertical wave numbers of a given problem. 

Approximating an IVP of \eqref{eq:hevitesteq} with the method \eqref{eq:dirkIMEX}, initial value $u(t_0) = u_0$, and step-size $\Delta t > 0$ results in the following difference equation
$$u_{m+1} = R_H(\Delta t k_x,\Delta t k_z )u_m, \quad m \in \{0\}\cup \mathbb{N},$$
where the stability matrix $R_H$ is defined by
\begin{equation}\label{eq:Hstabfun}
R_H(x,z) = I_3 -i(b^T \otimes xN + \hat{b}^T \otimes zS)(I_{3r} +  A \otimes ixN + \hat{A} \otimes iz S)^{-1}(\mathbf{1}_r \otimes I_3),
\end{equation}
where $I_w$ is the $w \times w$ identity matrix ($w \in \mathbb{N}$), $\mathbf{1}_w := (1,\ldots,1)^T\in \mathbb{R}^w$, and $\otimes$ represents the Kronecker product.  The HEVI or H-stability region is defined as
$$\mathcal{S}_{H} := \{x,z \geq 0: \text{ each eigenvalue of } R_H(x,z) \text{ is at most 1 in modulus}\}.$$ 
Define the set $\mathcal{F}_{\Delta t} := \{(\Delta t k_x,\Delta t k_z): k_x \in K_x, k_z \in K_z\}$.  Stable time-steps $\Delta t$ are those for which $\mathcal{F}_{\Delta t}\subseteq  \mathcal{S}_H$.  We will use H-stability regions to improve the stability of methods we derive in Section \ref{sec:imkgstability}.  By considering $k_x = 0$ and $k_z = 0$ it follows that ensuring $\mathcal{F}_{\Delta t}\subseteq  \mathcal{S}_H$ for $\Delta t$ as large as possible requires that the explicit and implicit method of \eqref{eq:dirkIMEX} are each stable on the imaginary axis for $\Delta t$ as large as possible.


\section{Analysis and Formulation of the IMKG2 and IMKG3 methods} \label{sec:methods}

\subsection{Formulation}
For $q\in \mathbb{N}$ with $q \geq 2$, consider an FSAL $(q+1)$-stage IMEX RK method given by
\begin{small}
\begin{equation}\label{eq:IMEXrk_ls}
\def\arraystretch{1.2}
\begin{array}{c|cccccc}
0 &       &         &        &         &         &\\
c_1       & \alpha_1     &        &         &         &\\
\vdots    & \beta_1     & \ddots &         &         &\\
          & \vdots  &        & \alpha_{q-2} &         &\\
\vdots          & \vdots  &        &         & \alpha_{q-1} &\\ 
c_q       & \beta_{q-1} &        &         &         & \alpha_q & \\ \hline
          & \beta_{q-1} &        &         &         & \alpha_q &

\end{array} \quad \def\arraystretch{1.2}\begin{array}{c|cccccc}
0 & &     &                  &           &               &\\
\hat{c}_1 & \hat{\alpha}_1        & \hat{d}_1 &               &                 &\\
\vdots    & \hat{\beta}_1        & \ddots    & \ddots        &                &\\
          & \vdots           &           & \hat{\alpha}_{q-2} & \hat{d}_{q-2} &     \\
\vdots      &  \vdots          &           &               & \hat{\alpha}_{q-1}  & \hat{d}_{q-1}\\ 
\hat{c}_q & \hat{\beta}_{q-1}    &           &               &                & \hat{\alpha}_q &  \\ \hline
          &   \hat{\beta}_{q-1}  &           &               &                & \hat{\alpha}_q & 

\end{array}
\end{equation}
\end{small}

with $c:=A\textbf{1}_{q+1}$ and $\hat{c}:=\hat{A}\textbf{1}_{q+1}$ ($\mathbf{1}_{q+1}$ defined as in Section \ref{sec:imexstability}).  Define $\alpha := (\alpha_1,\hdots,\alpha_q)^T$, $\hat\alpha := (\hat\alpha_1,\hdots,\hat\alpha_q)^T$, $\beta := (\beta_1,\hdots,\beta_{q-1})^T$,  $\hat\beta := (\hat\beta_1,\hdots,\hat\beta_{q-1})^T$, $\hat\delta := (\hat d_1,\hdots,\hat d_{q-1})^T$.  Any entry in the above Butcher tableau not corresponding to one of $\alpha$, $\hat\alpha$, $\beta$, $\hat\beta$, or $\hat\delta$ is set to zero.  The methods \eqref{eq:IMEXrk_ls} are three-register methods because each stage depends on at most three stages (registers correspond to the number of vectors that must be stored in memory within a time-step).   These methods become two-register methods when $\beta_j = \hat{\beta}_j = 0$ for $j=1,\ldots,q-1$.

\subsection{Accuracy}

Well-known results on polynomial interpolation imply that the order $p$ of any RK method $\def\arraystretch{1.2}\begin{array}{c|c} c & A  \\ \hline & b^T \\ \end{array}$ where $b$ has $l$ nonzero entries satisfies the bound $p \leq 2l$.  Consequently,  methods of the form \eqref{eq:IMEXrk_ls} are at most fourth order accurate.  We focus on second and third order accuracy since fourth and higher order accuracy requires that the method coefficients satisfy at least 52 additional equations \cite[pp. 314-315]{JV2000}.   This restricts the number of free coefficients available for improving stability and efficiency properties unless $q$ is sufficiently large ($q > 6$).

The following theorem gives simplified criteria for methods of the form \eqref{eq:IMEXrk_ls} to be second or third order accurate. 
\begin{theorem}\label{thm:orderthm}
The following two conclusions hold (with the convention that $\alpha_{k}$, $\hat\alpha_{k}$, $\beta_{k}$, $\hat\beta_{k}$, $\hat{d}_{k}$ equal $0$ when $k \leq 0$). 
\begin{enumerate}

\item A method \eqref{eq:IMEXrk_ls} with $q \geq 2$ is second order accurate if and only if the method coefficients satisfy
\begin{equation}\label{eq:orderthm0.1}
\left\{\begin{array}{c}
\alpha_q(\beta_{q-2}+\alpha_{q-1}) = \alpha_q(\hat{\beta}_{q-2}+\hat{\alpha}_{q-1} + \hat{d}_{q-2}) =1/2\\
\hat{\alpha}_q(\hat{\beta}_{q-2}+\hat{\alpha}_{q-1} + \hat{d}_{q-2}) = \hat{\alpha}_q(\beta_{q-2}+\alpha_{q-1}) = 1/2\\
 \alpha_q + \beta_{q-1} = 1 = \hat{\alpha}_q + \hat{\beta}_{q-1}.
\end{array}\right.\end{equation}
If $\beta_{q-1} = \hat{\beta}_{q-1} = \beta_{q-2} = \hat{\beta}_{q-2} = 0$, then this is equivalent to $\alpha_q = 1 = \hat{\alpha}_q$, $\alpha_{q-1} = 1/2$, and $\hat{\alpha}_{q-1} + \hat{d}_{q-2} = 1/2$.  

\item A method \eqref{eq:IMEXrk_ls} with $q \geq 2$ is third order accurate if and only if $\alpha_q = 3/4 = \hat{\alpha}_q$, $\beta_{q-1} = 1/4 = \hat{\beta}_{q-1}$, and the remaining method coefficients satisfy: 
\begin{equation}\label{eq:3rdorderthm}
\left\{\begin{array}{c}
\hat{\alpha}_{q-1}(\hat{\alpha}_{q-2} + \hat{d}_{q-2} + \hat{\beta}_{q-3}) + 2\hat{d}_{q-1}/3 = 2/9 \\
\hat{\alpha}_{q-1}(\alpha_{q-2} + \beta_{q-3}) + 2 \hat{d}_{q-1}/3 = 2/9\\
\hat{\alpha}_{q-1}(\hat{\alpha}_{q-2} + \hat{d}_{q-2} + \hat{\beta}_{q-3}) = 2/9 = \alpha_{q-1}(\alpha_{q-2} + \beta_{q-3})\\
\hat{\alpha}_{q-1} + \hat{d}_{q-1} + \hat{\beta}_{q-2} = 2/3 = \alpha_{q-1} + \beta_{q-2}.
\end{array}\right.
\end{equation}
\end{enumerate}
\end{theorem}
\begin{proof}
To prove the first conclusion assume that $q \geq 2$.  A $(q+1)$-stage IMEX RK method \eqref{eq:dirkIMEX} is second order accurate if and only if \cite[pp. 314-315]{JV2000}:
\begin{equation}\label{eq:2ndorder}
b^T \bold{1}_{q+1} = 1 = \hat{b}^T \bold{1}_{q+1}, \quad b^T c = b^T \hat{c} = \hat{b}^T \hat{c} = \hat{b}^T c = 1/2.
\end{equation}
Substituting the double Butcher tableau \eqref{eq:IMEXrk_ls} into \eqref{eq:2ndorder} shows that Equation \eqref{eq:2ndorder} is equivalent to Equation \eqref{eq:orderthm0.1}.  This proves the first conclusion. 

To prove the second conclusion assume that $q \geq 2$ and define diagonal matrices $C := \text{diag}(0,c_1,\ldots,c_{q})$ and $\hat{C} := \text{diag}(0,\hat{c}_1,\ldots,\hat{c}_{q})$.  A $(q+1)$-stage IMEX RK method \eqref{eq:dirkIMEX} is third order accurate if and only if \cite[pp. 314-315]{JV2000}:
\begin{equation}\label{eq:3rdorder}
\left\{
\begin{array}{lcr}
b^T \bold{1}_{q+1} = 1 = \hat{b}^T \bold{1}_{q+1}, \quad b^T c = b^T \hat{c} = \hat{b}^T \hat{c} = \hat{b}^T c = 1/2\\
b^T A c = b^T A \hat{c} = b^T \hat{A} c = b^T \hat{A} \hat{c} = \hat{b}^T A c = \hat{b}^T A \hat{c} = \hat{b}^T \hat{A} c = \hat{b}^T \hat{A} \hat{c} =  1/6\\
b^T C c = b^T C \hat{c} = b^T \hat{C} c = b^T \hat{C} \hat{c} = \hat{b}^T C c = \hat{b}^T C \hat{c} = \hat{b}^T \hat{C} c = \hat{b}^T \hat{C} \hat{c} = 1/3.
\end{array}\right.
\end{equation}
Substituting the double Butcher tableau \eqref{eq:IMEXrk_ls} into \eqref{eq:3rdorder} shows that the method \eqref{eq:IMEXrk_ls} is third order accurate if and only if the following system of equations is satisfied for every $\alpha_l',\overline{\alpha}_l,\tilde{\alpha}_l \in \{\alpha_l,\hat{\alpha}_l\}$ ; $\beta_l',\overline{\beta}_l,\tilde{\beta}_l \in \{\beta_l,\hat{\beta}_l\}$; $d_l',\overline{d}_l,\tilde{d}_l \in \{0,\hat{d}_l\}$; and $l \in \{q-2,q-1,q\}$:
$$\overline{\beta}_{q-1} + \overline{\alpha}_q = 1 \quad \text{(Eq 1)}$$
$$\overline{\alpha}_q( \alpha_{q-1}' + d_{q-1}' + \beta_{q-2}') = 1/2 \quad \text{(Eq 2)}$$
$$1/6 = \overline{\alpha}_q \tilde{\alpha}_{q-1} (\alpha_{q-2}'+d_{q-2}' + \beta_{q-3}') +\overline{\alpha}_q \tilde{d}_{q-1} (\alpha_{q-1}' + d_{q-1}' + \beta_{q-2}') \quad \text{(Eq 3)}$$
$$1/3 = \overline{\alpha}_q (\tilde{\alpha}_{q-1} + \tilde{d}_{q-1} + \tilde{\beta}_{q-2})(\alpha_{q-1}' + d_{q-1}' + \beta_{q-2}') \quad \text{(Eq 4)}.$$
We say that every version of Eq $k$, $k \in \{\text{1},\text{2},\text{3},\text{4}\}$,  is satisfied if it is satisfied for every $\alpha_l',\overline{\alpha}_l,\tilde{\alpha}_l \in \{\alpha_l,\hat{\alpha}_l\}$ ; $\beta_l',\overline{\beta}_l,\tilde{\beta}_l \in \{\beta_l,\hat{\beta}_l\}$; $d_l',\overline{d}_l,\tilde{d}_l \in \{0,\hat{d}_l\}$; and $l \in \{q-2,q-1,q\}$.  Substituting Eq 2 into Eq 4 implies that
\begin{equation}\label{eq:orderthm2}
\alpha_{q-1} + d_{q-1} + \beta_{q-2} = 2/3 = \hat{\alpha}_{q-1} + \hat{d}_{q-1} + \hat{\beta}_{q-2}.
\end{equation}
Eq 1-2 together with \eqref{eq:orderthm2} then imply that
\begin{equation}\label{eq:orderthm3}
\alpha_q = 3/4 = \hat{\alpha}_q, \quad \beta_{q-1} = 1/4 = \hat{\beta}_{q-1}.
\end{equation}
On the other hand, if \eqref{eq:orderthm2} \& \eqref{eq:orderthm3} are satisfied then so are all versions of Eq 1, Eq 2, and Eq 4.  Thus, \eqref{eq:orderthm2} \& \eqref{eq:orderthm3} are satisfied if and only if all versions of Eq 1, Eq 2, and Eq 4 are satisfied.  Substituting \eqref{eq:orderthm2}-\eqref{eq:orderthm3} into Eq 3 results in 
\begin{equation}\label{eq:orderthm4}
\tilde{\alpha}_{q-1} (\alpha_{q-2}'+d_{q-2}' + \beta_{q-3}') +  2 \tilde{d}_{q-1}/3 = 2/9.
\end{equation}
It then follows that every version of Eq 1-4 being satisfied is equivalent to \eqref{eq:orderthm2}-\eqref{eq:orderthm4}.  Therefore, every version of Eq 1-4 being satisfied is equivalent to \eqref{eq:3rdorderthm} and \eqref{eq:orderthm3}.  This completes the proof of the second conclusion.
\end{proof}

Theorem \ref{thm:orderthm} only constrains $\{\alpha_{j},\beta_j,\hat{\alpha}_j,\hat{\beta}_j,\hat{d}_j :j=l,\ldots,q\}$ where $l=q-1$ for second order accuracy or $l=q-2$ for third order accuracy.  The remaining coefficients can be chosen to improve stability properties or reduce the number of implicit stages.

\subsection{Stability basics}\label{sec:imkgstability}

The stability polynomial for the explicit method of \eqref{eq:IMEXrk_ls} is

\begin{equation}\label{eq:expoly}
P(z) = 1 + \sum_{k=1}^{q}\left[ \left(\prod_{j=0}^{k-2} \alpha_{q-j}\right)(\alpha_{q-k+1}+\beta_{q-k})\right] z^k
\end{equation}
with the convention that $\prod_{j=0}^{-1} \alpha_{q-j} = 1$ and $\beta_0 = 0$.  The formulas in \eqref{eq:expoly} defining the coefficients of $P(z)$ can then be coupled with the order conditions from Theorem \ref{thm:orderthm} to derive accurate IMEX RK methods where the explicit method has a KGO or KGNO stability polynomial.

Let $\hat{R}(z) = \hat{P}(z)/\hat{Q}(z)$ be the stability function of the implicit method of \eqref{eq:IMEXrk_ls}, where $\hat{Q}(z) = \prod_{j=1}^{q-1}(1-z\hat{d}_j)$, $\hat{P}(z) :=1 + \sum_{j=1}^{q}\hat{\sigma}_j z^j$, and
\begin{align}\label{eq:ap1}
\hat{\sigma}_1 = & \text{ } \hat{\alpha}_q + \hat{\beta}_{q-1}-\sum_{j=1}^{q-1}\hat{d}_j, \\
\hat{\sigma}_2 = & \text{ } \hat{\alpha}_q (\hat{\alpha}_{q-1} + \hat{\beta}_{q-2}) -\hat{\alpha}_q \sum_{j=1}^{q-2}\hat{d}_j - \sum_{j=1}^{q-1} \hat{\beta}_{q-1} \hat{d}_j + \sum_{j\neq k \leq q-1} \hat{d}_j \hat{d}_k, \label{eq:ap2}\\
\hat{\sigma}_3 = & \text{ } \hat{\alpha}_q(\hat{\alpha}_{q-1}\hat{\alpha}_{q-2} + \hat{\beta}_{q-3}) -  \hat{\alpha}_{q-1}\hat{\alpha}_q \sum_{j=1}^{q-3}\hat{d}_j - \sum_{j=1}^{q-2} \alpha_q \beta_{q-2} d_j \label{eq:ap3} \\
&\quad + \hat{\alpha}_q \sum_{j\neq k \leq q-2} \hat{d}_j \hat{d}_k + \sum_{j,k \leq q-1} \hat{\beta}_{q-1} d_j d_k - \sum_{j\neq k \neq l \leq q-1}\hat{d}_j \hat{d}_k \hat{d}_l. \nonumber
\end{align}

The following theorems give sufficient conditions for I-, A-, and L-stability of the implicit method of \eqref{eq:IMEXrk_ls}.   Note that I-stability is a prerequisite for IMEX RK methods to have H-stability regions with desirable properties.
\begin{theorem}\label{thm:Lstable}
Let $k\in \{1,\ldots,q\}$ be such that $\hat{\sigma}_j = 0$ for $j=k,\ldots,q$.  Then $\hat{R}(z) \rightarrow 0$ as $|z|\rightarrow \infty$ if and only if the implicit method has at least $k$ nontrivial implicit stages.
\end{theorem}
\begin{proof}
Since $\hat{\alpha}_j = 0$ for $j=k,\ldots,q$ it follows that $\text{deg}(\hat{P}(z)) = k-1$.  We have $\hat{R}(z) \rightarrow 0$ as $|z| \rightarrow \infty$ if and only if $\text{deg}(\hat{Q}(z)) > \text{deg}(\hat{P}(z))$ which is the case if and only if the implicit method has at least $k$ nontrivial implicit stages.
\end{proof}
We refer to a method that satisfies $\hat{R}(\infty) = 0$ (i.e. $\hat{R}(z) \rightarrow 0$ as $|z| \rightarrow \infty$) as a VI method (vanishes at infinity method).  We use this term to distinguish methods that are I-stable, not L-stable, but still satisfy $\hat{R}(\infty)=0$.  
\begin{theorem}\label{thm:Istable}
Suppose that $\hat{Q}(iy) \neq 0$ for all $y \in \mathbb{R}$.  Assume $\hat{\sigma}_4,\ldots,\hat{\sigma}_{q}= 0$ and let $\gamma_1$, $\gamma_2$, and $\gamma_3$ be given by
$$\gamma_1 := \hat{\sigma}_1^2 - 2\hat{\sigma}_2 - \sum_{j=1}^{q-1} d_j^2, \quad \gamma_2 :=  \hat{\sigma}_2^2 -2 \hat{\sigma}_1 \hat{\sigma}_3 - \sum_{j\neq k}\hat{d}_j \hat{d}_k, \quad \gamma_3 := \hat{\sigma}_3^2 - \sum_{j \neq k \neq l} \hat{d}_j^2 \hat{d}_k^2 \hat{d}_l^2.$$
 If $\gamma_1,\gamma_2,\gamma_3 \leq 0$, then the implicit method of \eqref{eq:IMEXrk_ls} is I-stable.  If in addition $\hat{d}_j \geq 0$ for $j=1,\ldots,q-1$, then the implicit method of \eqref{eq:IMEXrk_ls} is A-stable.
\end{theorem}
\begin{proof}
An RK method with $\hat{Q}(iy) \neq 0$ for all $y \in \mathbb{R}$ is I-stable if and only if $|\hat{P}(iy)|^2 - |\hat{Q}(iy)|^2 \leq 0$ for all $y \in \mathbb{R}$.  From $\hat{\sigma}_4 = \ldots = \hat{\sigma}_q = 0$ it follows that 
 $$|\hat{P}(iy)|^2 = 1 + (\hat\sigma_1^2 - 2\hat\sigma_2)y^2 + (\hat\sigma_2^2 -2 \hat\sigma_1 \hat\sigma_3)y^4 + \hat\sigma_3^2 y^6$$
 $$ |\hat{Q}(iy)|^2 = 1 + \sum_{j=1}^{q-1}\hat{d}_j^2 y^2 + \sum_{j\neq k}^{q-1} \hat{d}_j^2 \hat{d}_k^2 y^4 + \sum_{j \neq k \neq l}^{q-1} \hat{d}_j \hat{d}_k\hat{d}_l y^6 + h(y)y^8$$ 
where $h(y)$ is an even polynomial in $y$ with positive coefficients.  Therefore 
$$|\hat{P}(iy)|^2-|\hat{Q}(iy)|^2 = \gamma_1 y^2 + \gamma_2 y^4 + \gamma_3 y^6 - \xi(y)y^8 \leq \gamma_1 y^2 + \gamma_2 y^4 + \gamma_3 y^6, \quad y \in \mathbb{R}.$$
I-stability of the implicit method of \eqref{eq:IMEXrk_ls} follows if $\gamma_1,\gamma_2,\gamma_3 \leq 0$.  The conclusion on A-stability follows from the maximum principal and the fact that $\hat{R}(z)$ is holomorphic on $\{z \in \mathbb{C}:\text{Re}(z) < 0\}$ if and only if $\hat{d}_1,\ldots,\hat{d}_{q-1}\geq 0$.
\end{proof} 
Equations \eqref{eq:ap1}-\eqref{eq:ap3} are used to ensure the hypotheses of Theorem \ref{thm:Istable} are satisfied.

\subsection{Definition, derivation, and H-stability of IMKG2-3 methods}\label{sec:Hstability}

We first define the IMKG2-3 methods:
\begin{definition}\label{def:imexkg}
  An IMKG2 method \eqref{eq:IMEXrk_ls} has $\beta_j = \hat{\beta}_j = 0$ for $j=1,\ldots,q-1$, is second order accurate, has an explicit method with a KGO or KGNO stability polynomial, and an implicit method that is I-stable.  An IMKG3 method \eqref{eq:IMEXrk_ls} is third order accurate, has an explicit method with a KGNO stability polynomial, and an implicit method that is I-stable.
\end{definition} 
Note that \eqref{eq:expoly} and the definition of KGO and KGNO polynomials implies that IMKG2-3 methods must have $\alpha_j,\hat\alpha_j \neq 0$ for $j=1,\hdots,q$.  To demonstrate the construction of an IMKG2-3, we provide the following example:
\begin{examp}\label{ex:imkg343a}
  We construct an IMKG3 method with $q = 4$.  The KGNO polynomial with $q = 4$ is $P(z) = 1 + z + z^2/2 + z^3/6 + z^4/24$.  Equation \eqref{eq:expoly} implies that $\alpha_1 = 1/(24 \alpha_2 \alpha_3 \alpha_4)$.  Third order accuracy (Theorem \ref{thm:orderthm}) and enforcement of $\hat{\beta}_1 = \beta_1$ and $\hat{\beta}_2 = \beta_2$ require that $\alpha_4 = \hat{\alpha}_4 = 3/4$, $\beta_3 = \hat{\beta}_3 = 1/4$, and
  $$\hat{\alpha}_{3} + \hat{d}_{3} + \beta_2 = 2/3 = \alpha_{3} + \beta_{2}, \quad \alpha_{3}(\hat{\alpha}_{2} + \hat{d}_{2} + \beta_1) = 2/9 = \alpha_{3}(\alpha_2 + \beta_1)$$
  $$\hat{\alpha}_{3}(\hat{\alpha}_{2} + \hat{d}_{2} + \beta_1) + 2\hat{d}_{3}/3 = 2/9 = \hat{\alpha}_{3}(\alpha_2+\beta_1) + 2 \hat{d}_{3}/3.$$
  Assume that $\hat{d}_2$, $\hat{d}_3$, $\beta_1$, and $\alpha_2$ have been specified.  We then set 
  $$\alpha_3 = \frac{2}{9(\alpha_2 + \beta_1)}, \quad \beta_2 = 2/3-\alpha_3, \quad \hat{\alpha}_3 = \frac{2/9-2\hat{d}_3/3}{\alpha_2+\beta_1},\quad \beta_2 = 2/3-\hat{d}_3-\hat{\alpha}_3$$
   to ensure third order accuracy.  Enforcing $\text{deg}(\hat{P}(z)) = 2$ requires that
  \begin{equation*}
    \small{\hat{d}_1 = \frac{ \hat{\alpha}_3\hat{\alpha}_4\hat{\beta}_1 + \hat{\beta}_3\hat{d}_2\hat{d}_3 -\hat{\alpha}_2\hat{\alpha}_3\hat{\alpha}_4  -\hat{\alpha}_4\hat{\beta}_2\hat{d}_2}{\hat{\alpha}_4\hat{d}_2+ \hat{\beta}_3\hat{d}_2 + \hat{\beta}_3 \hat{d}_3- \hat{d}_2\hat{d}_3- \hat{\alpha}_3\hat{\alpha}_4 - \hat{\alpha}_4\hat{\beta}_2 }, \quad \hat{\alpha}_1 = \frac{ \hat{\alpha}_3\hat{\alpha}_4\hat{\beta}_1\hat{d}_1- \hat{\alpha}_4\hat{\beta}_2\hat{d}_1\hat{d}_2+ \hat{\beta}_3\hat{d}_1\hat{d}_2\hat{d}_3}{\hat{\alpha}_2\hat{\alpha}_3\hat{\alpha}_4}.}
  \end{equation*}
  The values of $\hat{d}_2$, $\hat{d}_3$, $\beta_1$, and $\alpha_2$ can be chosen so that the implicit method is I- or A-stable and improve the H-stability region.  The choice $\hat{d}_3 = \hat{d}_2 = 1$, $\alpha_2 = 2/3$, and $\beta_1 = 1/3$ results in the IMKG343a method.  Theorems \ref{thm:Lstable} and  \ref{thm:Istable} imply IMKG343a is I-stable and a VI method.
\end{examp}
In the appendix we give the double Butcher tableaux for several IMKG2 methods with $q = 3,4,5$ (Table \ref{tab:imkg2methods}) and several IMKG3 methods with $q = 4,5$ (Table \ref{tab:imkg3methods}). 

To demonstrate how $\hat{d}_2$, $\hat{d}_3$, $\beta_1$, and $\alpha_2$ in Example \ref{ex:imkg343a} might be chosen to improve the H-stability region, we focus on spatially-discrete, hyperbolic-type PDEs on a bounded domain.  For a given spatial resolution, define $M_l := \max K_l$ for $l=x,z$ ($K_x$ and $K_z$ defined as in Section \ref{sec:imexstability}).  Note that $M_l < \infty$ and that there exists $m_x,m_z >0$, independent of resolution, such that $0 < m_l \leq \min K_l$ for $l=x,z$.  We define the vertical-to-horizontal aspect ratio $\chi := M_z/M_x$.  In our target application HOMME-NH (Section \ref{sec:nh}), we anticipate using about a 1km vertical and $\geq$3km horizontal resolutions, which results in $\chi > 2$. 


Consider an $r$-stage IMKG2-3 method with H-stability region $\mathcal{S}_H$.  Let $r_0 := r-1$ if the method has a KGO stability polynomial and $r_0 := \sqrt{(r-1)^2-1}$ if the method has a KGNO stability polynomial.  For $\gamma, n_0 > 0$ define
$$\mathcal{T}_{n_0} = \{(x,z) : z \geq 0, x\in [0,n_0] \},\text{ } \mathcal{E}_{\gamma,n_0} :=\{ ( x,z): x \in [0,n_0], z = 0 \text{ or } z \geq \gamma x  \} \cap \mathcal{T}_{n_0}.$$
These regions are illustrated in Figure \ref{fig:stab_sets} for $\mathcal{E}_{1/3,7/2}$ and $\mathcal{T}_{7/2}$.  Ideally, a method will have $\mathcal{T}_{r_0} \subseteq \mathcal{S}_H$, as is the case for IMKG232b (see Figure \ref{fig:Hstabregion_imkg232}).  In this case, $\mathcal{F}_{\Delta t } \subseteq \mathcal{T}_{r_0} \subseteq \mathcal{S}_H$ ($\mathcal{F}_{\Delta t}$ defined as in Section \ref{sec:imexstability}) for $\Delta t > 0$ such that $\Delta t M_x \in (0,r_0)$.  Thus, $\mathcal{T}_{r_0} \subseteq \mathcal{S}_H$ implies that stable time-steps are completely determined by the explicit method (via $r_0$) and the horizontal wave numbers (via $M_x$). 

Not all IMKG2-3 methods satisfy $\mathcal{T}_{r_0} \subseteq \mathcal{S}_H$ (see Figure \ref{fig:252stabdemo}).   It is also unclear what conditions on the method coefficients ensure that this containment holds.  However, methods where $\mathcal{T}_{r_0} \nsubseteq \mathcal{S}_H$, such as IMKG252b, can still be efficient for the values of $M_l$ and $m_l$ we expect in our target application HOMME-NH (see Table \ref{tab:maxusablestep} and Figures \ref{fig:runtime1}-\ref{fig:runtime2}).  This motivates describing a sub-optimal case where $\mathcal{T}_{r_0} \not\subseteq \mathcal{S}_H$.  Assume that $\gamma \in (0,m_z/M_x)$, $n_0 \approx r_0$, and $\mathcal{E}_{\gamma,n_0} \subseteq S_H$.  If $\Delta t > 0$ is such that $\Delta t M_x \in (0,n_0)$, then $\mathcal{F}_{\Delta t } \subseteq \mathcal{E}_{\gamma,n_0} \subseteq \mathcal{S}_H$. Thus, in this sub-optimal case, stable time-steps are still determined by $n_0 \approx r_0$ and $M_x$ if $\gamma \in (0,m_z/M_x)$.  

With this analysis in place, we have a methodology for choosing free method coefficients to improve the H-stability region of IMKG2-3 methods.  We first search for coefficients such that $\mathcal{T}_{r_0} \subseteq \mathcal{S}_{H}$.  If this fails, we then enforce $\mathcal{E}_{\gamma,n_0} \subseteq \mathcal{S}_{H}$ for $\gamma \geq 0$ and $r_0-n_0 \geq 0$ as small as possible.  Note that $m_z$ depends on the depth and scale height of the model atmosphere.  The value of $M_x$ depends on the horizontal resolution and is affected by our use of hyper-diffusion to stabilize high frequency modes arising from the spectral element discretization.  This makes it difficult to determine exactly how small $\gamma$ must be chosen so that $\mathcal{E}_{\gamma,n_0} \subseteq \mathcal{S}_{H}$ for $n_0 \approx r_0$.

Despite this drawback our analysis is still be predictive in many cases.  Consider Figures \ref{fig:Hstabregion_imkg232}-\ref{fig:Hstabregion_imkg242}.  If $\chi >> 1$, then these figures and our analysis predict that the maximum stable time-step of IMKG232a should be about 50\% of that of IMKG232b and that the maximum stable time-step of IMKG242a should be about 60\% that of IMKG242b.  If $\alpha \approx 1$, then we would predict IMKG232a-b and IMKG242a-b to have nearly equal maximum stable time-steps.  These predictions agree with the empirical results in Table \ref{tab:maxusablestep} except in the small planet $\times 1$ run for the IMKG242b method, which terminated from a solver failure rather than an instability for time-steps larger than 225 seconds.

\begin{figure}
\centering
\includegraphics[scale=0.4]{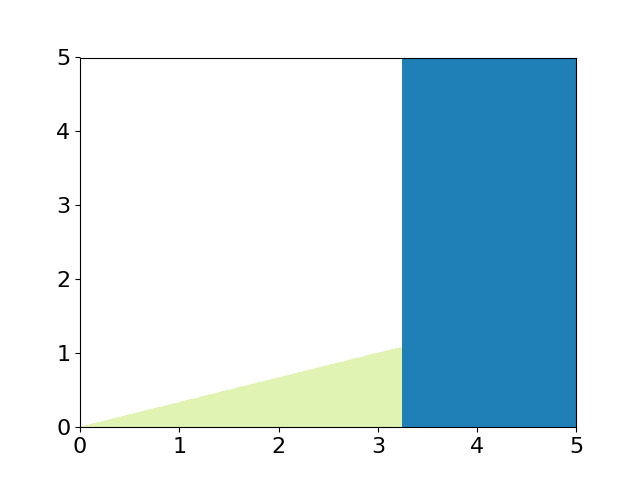}
\caption{The set $\mathcal{E}_{1/3,7/2}$ (dark blue region) and $\mathcal{T}_{7/2}$ (union of the unshaded white and light green regions). }.
\label{fig:stab_sets}
\end{figure}

\begin{figure}
\centering
\includegraphics[scale=0.4]{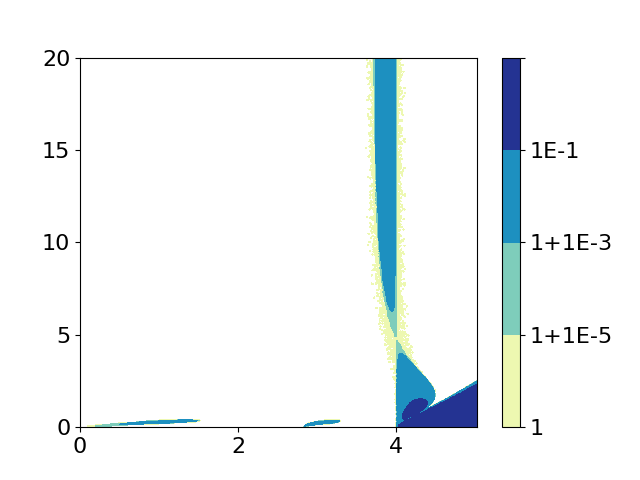}\includegraphics[scale=0.4]{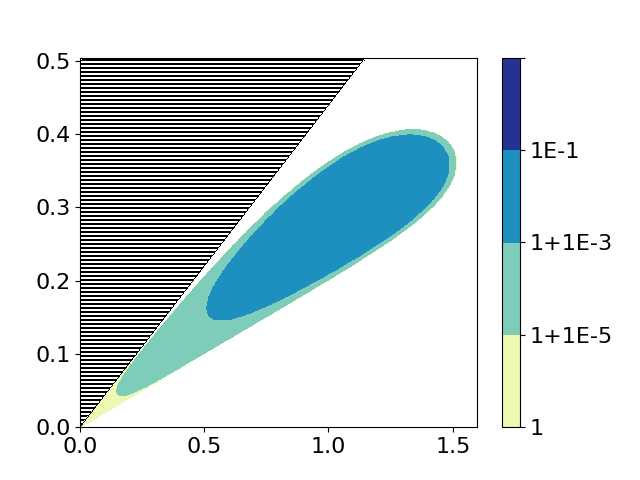}
\caption{H-stability region of IMKG252a ($z$ vs $x$ where the unshaded white and striped regions represent the H-stability region).  The double Butcher tableau of IMKG252b is given in Table \ref{tab:imkg2methods}. The striped region denotes the region above the line $z = \gamma x$ and the H-stability region contains $E_{3.5,\gamma}$ where $\gamma \approx .45$.   The blue and yellow shaded region denotes the modulus of the largest eigenvalue of the stability matrix $R_H$ when this modulus exceeds $1$.   }
\label{fig:252stabdemo}
\end{figure}

 \begin{figure}
\centering
\includegraphics[scale=0.4]{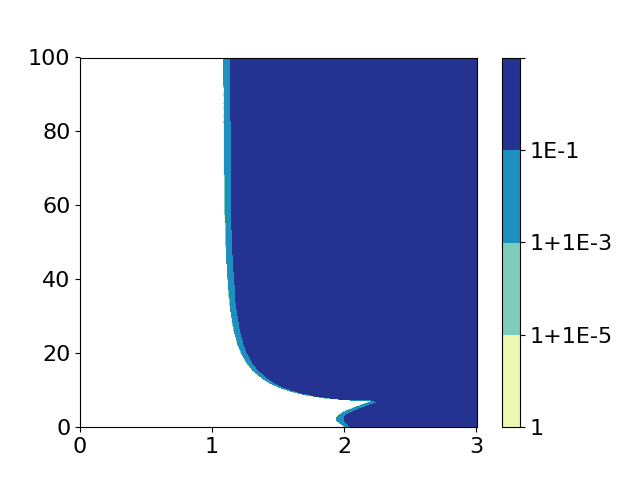}\includegraphics[scale=0.4]{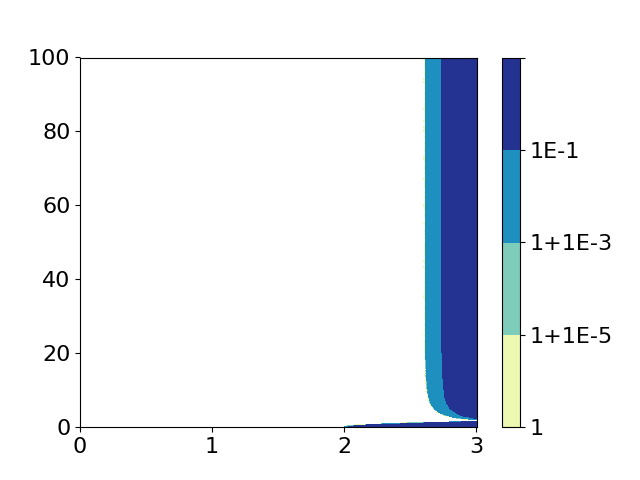}
\caption{H-stability regions ($z$ vs $x$ where the unshaded white region represents the H-stability region) of the IMKG232a (left) and IMKG232b (right) methods (double Butcher tableaux given in Table \ref{tab:imkg2methods}).  The blue and yellow shaded region denotes the modulus of the largest eigenvalue of the stability matrix $R_H$ when this modulus exceeds $1$. }
\label{fig:Hstabregion_imkg232}
\end{figure}

 \begin{figure}
\centering
\includegraphics[scale=0.4]{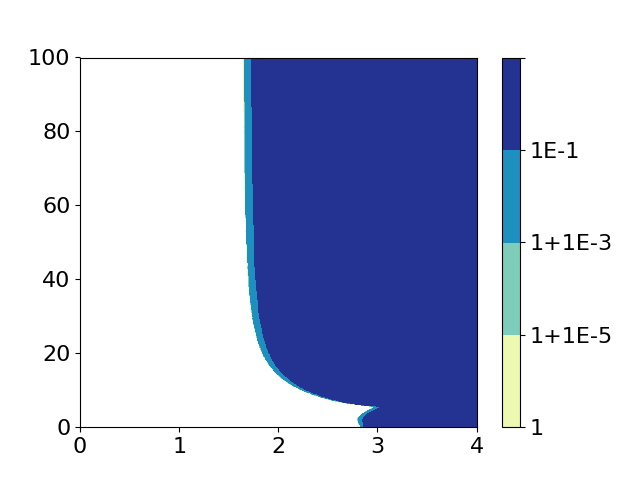}\includegraphics[scale=0.4]{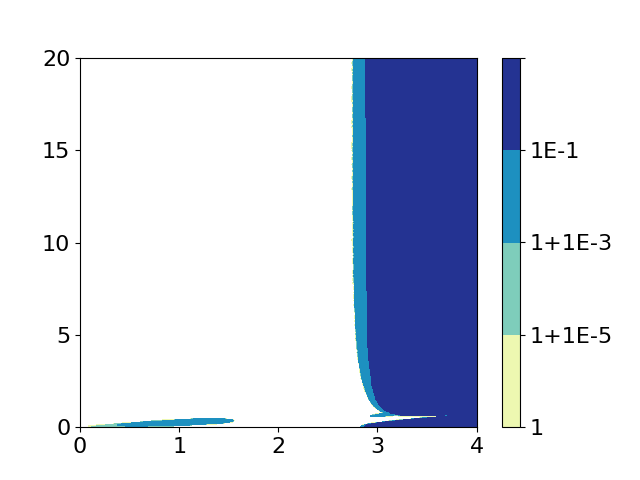}
\caption{H-stability regions ($z$ vs $x$ where the unshaded white region represents the H-stability region) of the IMKG242a (left) and IMKG242b (right) methods (double Butcher tableaux given in Table \ref{tab:imkg2methods}).   The blue and yellow shaded region denotes the modulus of the largest eigenvalue of the stability matrix $R_H$ when this modulus exceeds $1$. }
\label{fig:Hstabregion_imkg242}
\end{figure}


\section{The HOMME-NH nonhydrostatic dynamic core}\label{sec:nh}

In this section we introduce the HOMME-NH nonhydrostatic atmosphere model and its horizontally explicit vertically implicit (HEVI) partitioning.  HOMME-NH is a more realistic test bed for the performance of the IMKG2-3 methods for HEVI partitionings than low-dimensional test equations such as \eqref{eq:hevitesteq}.

\subsection{Formulation of HOMME-NH}\label{sec:hommenh}

A comprehensive derivation of HOMME-NH is given in \cite{hommenh}.  It is essentially a variant of the Laprise formulation \cite{Laprise1992}, where the shallow atmosphere and traditional approximations, defined as in \cite{Vallis2017},  are made.  The governing equations of HOMME-NH are given by

\begin{equation}\label{eq:nhtheta}
\arraycolsep=1.2pt\def\arraystretch{1.5}
\left\{\begin{array}{lcr}
\bo{u}_t + (\nabla_{\eta} \times \bo{u} + 2\Omega)\times \bo{u} + \frac{1}{2}\nabla_{\eta}(\bo{u} \cdot \bo{u}) + \dot{\eta} \pd{\bo{u}}{\eta}{} + \frac{1}{\rho}\nabla_{\eta} p + \mu \nabla_{\eta}\phi = 0, \quad \dot{\eta} := d\eta/dt\\

w_t + \bo{u} \cdot \nabla_{\eta} w + \dot{\eta} \pd{w}{\eta}{} + \mathfrak{g}(1- \mu) = 0, \quad \mu := \pd{p}{\eta}{}/\pd{\pi}{\eta}{} \\

\phi_t + \bo{u} \cdot \nabla_{\eta} \phi + \dot{\eta} \pd{\phi}{\eta}{} - \mathfrak{g}w  = 0 \\

\Theta_t + \nabla_{\eta} \cdot (\Theta \bo{u}) + \frac{\partial}{\partial \eta}(\Theta \dot{\eta})  = 0, \quad \Theta := \pd{\pi}{\eta}{} \theta_v \\

\frac{\partial}{\partial t} (\pd{\pi}{\eta}{}) + \nabla_{\eta} \cdot (\pd{\pi}{\eta}{} \bo{u}) + \frac{\partial}{\partial \eta} \left(\pd{\pi}{\eta}{} \dot{\eta}\right) = 0. \\
\end{array}
\right.\end{equation}
The horizontal spatial variables lie on a spherical domain, while the vertical coordinate $\eta$ is the mass-based hybrid terrain-following coordinate introduced in \cite{Kasahara1974}, with $\eta=1$ representing the model surface and $\eta=\eta_\text{top}$ the model top.   
The vector $\bo{v} = (u,v,w)^T$ is the fluid velocity with $\bo{u}:=(u,v)^T$, $\theta_v$ is the virtual potential temperature, $\mathfrak{g}$ is the gravitational constant, $\phi = \mathfrak{g}z$ is the geopotential, $\rho$ is the fluid density, $p$ is the pressure, $2\Omega \times \bo{u}$ is the Coriolis term with rotation rate $\Omega$, and the symbol $\nabla_{\eta}$ represents the horizontal gradient with respect to $\eta$.  The variable $\pi$ represents the hydrostatic pressure defined so that \begin{equation}\label{eq:hydroapprox}
\pd{\pi}{z}{} = -\rho \mathfrak{g}
\end{equation}
with the boundary condition $\pi = \pi_\text{top}$ imposed at $\eta=\eta_\text{top}$ for some constant $\pi_\text{top}$.  Note that if $\mu \equiv 1$ is enforced, then \eqref{eq:nhtheta} becomes a nonstiff hydrostatic model without vertically propagating acoustic waves.  We then say the model is in hydrostatic mode; otherwise it is in nonhydrostatic mode.  In hydrostatic mode \eqref{eq:nhtheta} can be efficiently integrated by an explicit RK method with a KGO or KGNO stability polynomial and the CFL condition is controlled by the stability of the KGO and KGNO polynomials on the imaginary axis.  This CFL condition places an upper bound on the maximum stable step-size for IMEX RK methods integrating \eqref{eq:nhtheta} in nonhydrostatic mode.  

\subsection{Analysis of vertical acoustic wave propagation}\label{sec:vertwav}

HEVI partitioning is commonly employed in nonhydrostatic models \cite{Satoh2002,LWW2013,ARKODE2018,accerlatenuma}.  Our HEVI strategy partitions \eqref{eq:nhtheta} into a stiff term representing vertical acoustic wave propagation and a nonstiff term representing advection and horizontal acoustic wave propagation.  Devising such a partitioning of \eqref{eq:nhtheta} requires understanding the structure of vertically propagating waves.  In this section, we analyze this structure in our Laprise-like formulation. 

Due to the mass-based vertical coordinate, oscillations in density will cause oscillations in $\phi$ \cite[Appendix A]{Laprise1992}.  In particular, density oscillations from vertical acoustic waves manifest in the physical position of the model $\eta$-layers and are decoupled from vertical motions relative to this moving coordinate system.  Therefore, the vertical advection terms (e.g. $\dot{\eta}\pd{w}{\eta}{}$) are not associated with the fast motions of the vertical acoustic waves.  This isolates the vertical acoustic waves to the two non-transport terms in the equations for $w$ and $\phi$ in \eqref{eq:nhtheta}.  To see this, consider
\begin{equation}\label{eq:vertwav}
w_t - \mathfrak{g} (1-\mu) = 0, \quad \phi_t - \mathfrak{g} w = 0, \quad
\pd{\phi}{\eta}{}= -R \Theta p^{\kappa - 1}, \quad \rho = - \pd{\pi}{\eta}{}/\pd{\phi}{\eta}{},
\end{equation}
where for simplicity we ignore moisture and drop the constant reference pressure in the definition of Exner pressure: $\Pi=p^{\kappa}$ rather than $\Pi = (p/p_0)^{\kappa}$ for some reference constant $p_0$.  The fourth equation $\rho = -\pd{\pi}{\eta}{} \left(\pd{\phi}{\eta}{} \right)^{-1}$ follows from the definition of $\pi$ and $\mu := \pd{p}{\eta}{}/\pd{\pi}{\eta}{}$ as in \eqref{eq:nhtheta}. Following \cite{Thuburn2012}, we linearize  \eqref{eq:vertwav} around the constant state $\phi_{ref}$, $w_{ref}$, $p_{ref}$,... with perturbations given by $\tilde{\phi}$, $\tilde{w}$,.. and  $\tilde{\Theta}=\tilde{\pi}=0$ to obtain:
$$\tilde{\phi}_{tt} - c^2\dfrac{\partial^2\tilde{\phi}}{\partial\eta^2}=0
\qquad
c^2 =\dfrac{RT_{ref}}{(\kappa-1)}\left(\pd{\eta}{z}{}\right)^2$$
where $c$ is the sound speed in $\eta$ coordinates.  With this in mind, we choose our HEVI partitioning such that $\mathfrak{g}(1-\mu)$ and $\mathfrak{g}w$ are the only implicitly treated terms of \eqref{eq:nhtheta}.

\subsection{HEVI partitioning and simplification of IMEX RK stage equations}\label{sec:imexsplitting}

We express \eqref{eq:nhtheta} as a general evolution equation
\begin{equation}\label{eq:evolutioneqn}
  \xi_t = f(\xi), \quad \xi = (u,v,w,\phi,\Theta,
  \partial \pi / \partial \eta)^T.
\end{equation}
We define the HEVI additive partitioning of \eqref{eq:nhtheta} with $f(\xi) = n(\xi)+s(\xi)$ as follows:
\begin{equation}\label{eq:nhtheta_IMEXsplitting}
s(\xi) := (0 ,0, -\mathfrak{g}(1-\mu), \mathfrak{g} w ,0 ,0)^T, \quad n(\xi) := f(\xi)-s(\xi).
\end{equation}
Consider the solution of the IVP $\xi_t = n(\xi)+s(\xi)$, $\xi(t_0) = \xi_0$ by an IMEX RK method of the form \eqref{eq:dirkIMEX} with step-size $\Delta t > 0$.  For $j=1,\hdots,r$ and $m\in \{0\} \cup \mathbb{N}$, we express the internal stages as $g_{m,j} = (g_{m,j}^u,g_{m,j}^v,g_{m,j}^w,g_{m,j}^{\phi},g_{m,j}^{\Theta},g_{m,j}^{\partial \pi})^T$ where $\partial \pi := \partial \pi / \partial \eta$.  Using the notation of \eqref{eq:dirkIMEX} we write
$$g_{m,j} = E_{m,j} + \Delta t \hat{A}_{j,j} s(g_{m,j}), \quad j=1,\hdots,r.$$
From the definition of $n$ and $s$, the internal stages for $u$, $v$, $\Theta$, and $\partial \pi / \partial \eta$ are explicit:
$$g_{m,j}^u = E_{m,j}^u, \quad g_{m,j}^v = E_{m,j}^v, \quad g_{m,j}^{\Theta} = E_{m,j}^{\Theta}, \quad g_{m,j}^{\text{dp}} = E_{m,j}^{\partial \pi}.$$
On the other hand, determining $g_{m,j}^w$ and $g_{m,j}^{\phi}$ requires solving the following system:
\begin{equation}\label{eq:IMEXstageeqns}
\left\{
\begin{array}{lcr}
g_{m,j}^w = E_{m,j}^w + \Delta t \mathfrak{g} \hat{A}_{j,j}(1-\mu_{m,j}) \\
g_{m,j}^{\phi} = E_{m,j}^{\phi} + \Delta t \mathfrak{g} \hat{A}_{j,j}g_{m,j}^w \\
\end{array}\right., \quad m \in \{0\} \cup \mathbb{N}, \quad j =1,\hdots,r,
\end{equation}
where $\mu_{m,j} := \mu(g_{m,j}^w,g_{m,j}^{\phi})$  (recall that $\mu := \pd{p}{\eta}{}/\pd{\pi}{\eta}{}$).  The second equation in \eqref{eq:IMEXstageeqns} is rearranged to 
\begin{equation}\label{eq:wphirelation}
g_{m,j}^w = (g_{m,j}^{\phi} - E_{m,j}^{\phi})/(\mathfrak{g}\Delta t \hat{A}_{j,j}).
\end{equation}
It follows that $g_{m,j}^w$ is an explicit function of $g_{m,j}^{\phi}$ and $\mu_{m,j} = \mu(g_{m,j}^{\phi})$.  Substituting \eqref{eq:wphirelation} into the first equation of \eqref{eq:IMEXstageeqns} implies that $g_{m,j}^{\phi}$ is given by
$$g_{m,j}^{\phi} - E_m^{\phi} = g \Delta t \hat{A}_{j,j}E_m^w - (\mathfrak{g}\Delta t \hat{A}_{j,j})^2 (1-\mu_{m,j}), \quad m \in \mathbb{N} \cup \{0\}, \quad j =1,\hdots,r.$$ 
Hence we can find $g_{m,j}^{\phi}$ by solving $G_{m,j}(g_{m,j}^{\phi})  = 0$ where
\begin{equation}\label{eq:internalphieqn}
G_{m,j}(g_{m,j}^{\phi}) = g_{m,j}^{\phi} - E_{m,j}^{\phi} - \mathfrak{g} \Delta t \hat{A}_{j,j}E_m^w + (\mathfrak{g}\Delta t \hat{A}_{j,j})^2 (1-\mu_{m,j}).
\end{equation}
We solve Equation \eqref{eq:internalphieqn} with Newton's method (described in Section \ref{sec:solver}).


\section{Implementation and experiments }\label{sec:experiments}

\subsection{Spatial discretization and implementation details}\label{sec:impdetails}

HOMME-NH is implemented in the High Order Method Modeling Environment (HOMME) \cite{camse2012,cam42013}.  Horizontal derivatives (those involving $\nabla_{\eta}$) are discretized with fourth order spectral elements \cite{FournierTaylor2010} on the cubed sphere grid \cite[Sec. 4]{ITT1997}.  Compatibility (see \cite{FournierTaylor2010}) of the spectral element method implies discrete conservation of mass, energy, and potential vorticity by the $\nabla_{\eta}$ operator in continuous time.  Vertical derivatives (those involving $\partial / \partial \eta$) are discretized with the second order SB81 Simmons and Burridge \cite{SimmonsBurridge1981} method with a Lorenz vertical staggering \cite{Lorenz1960}.   Compatibility of the spectral element and SB81 methods and careful treatment of the thermodynamic variables yields mass and energy conservation of the spatial discretization in continuous time (see \cite{hommenh} for details).


Spectral element discretizations can generate unstable oscillations  \cite{ullrichetal2018}.  These oscillations are dampened for long simulations using hyper-viscosity with respect to a hydrostatic background state in an operator split manner.  Starting with $\xi_m$, defined as the approximation to $\xi(t)$ (Section \ref{sec:imexsplitting}) at time $t_m$, an approximation to $\xi(t)$ at time $t_{m+1}$ is first formed by advancing a single time-step of the full space-time discretization of \eqref{eq:nhtheta}.  Hyper-viscosity is then applied to the result, denoted $\tilde{\xi}_{m+1}$, to obtain $\xi_{m+1}$:
$$\xi_{m+1} := \tilde{\xi}_{m+1} + \nu \Delta t \Delta^2_{\eta}( \tilde{\xi}_{m+1}-\xi_{m+1}^{\pi}), \quad \nu > 0,$$
where $\nu$ is determined by the grid scale of the horizontal spatial resolution and $\xi_{m+1}^{\pi}$ is a hydrostatic background state (see \cite{hommenh} for a more detailed description).  This operator splitting limits temporal accuracy to first order (unless the spatial and temporal resolutions are reduced simultaneously). Therefore, hyper-viscosity is not applied in our formal convergence study (Section \ref{sec:accuracy}). 

IMEX RK methods are implemented with an interface to the ARKode package \cite{ARKODE2018,Gardner2017} of the SUNDIALS library \cite{hindmarsh2005sundials}.  That interface, based off one for the nonhydrostatic Tempest dynamical core \cite{guerraullrich2016}, was developed in \cite{voglpaper2019} for rapid testing and implementation of IMEX RK methods, along with evaluation methodologies for accuracy, conservation, and efficiency.  We compare our IMKG2-3 methods with several IMEX RK methods from the literature (henceforth called the non-IMKG methods) as well as the five stage, third order accurate KGU35 explicit RK method \cite[Eq. 56]{guerraullrich2016}.  The non-IMKG methods we consider are ARS232, ARS343, and ARS443 \cite[Sec. 2.5,2.7,2.8]{ARS1997};  ARK324 and ARK346 \cite[p. 47-48]{CK2003}; and ARK2 \cite[Eq. 3.9]{GKC2013}.

\subsection{Solver implementation}\label{sec:solver}

We now describe the computation of the implicit stages $g_{m,j}$ from Section \ref{sec:imexsplitting} via Newton's method.  From the initial guess $g_{m,j}^{(0)} = E_{m,j}$, the ARKode package generates iterates $g_{m,j}^{(k+1)}$ of the form $g_{m,j}^{(k+1)} = g_{m,j}^{(k)} + \delta_{m,j}^{(k+1)}$, where $ \delta_{m,j}^{(k+1)}$ is the solution of
\begin{equation*}
  \big[I - \Delta t \hat{A}_{j,j}\partial_q s(g_{m,j}^{(k)})\big] \delta_{m,j}^{(k+1)} = E_{m,j}, \quad \partial_q s := \partial s/\partial q.  \\
\end{equation*}
Recall from Section \ref{sec:imexsplitting} that the only non-zero elements of $\partial_q s(g_{m,j}^{(k)})$ are those such that both the row and column pertain to $g_{m,j}^w$ or $g_{m,j}^\phi$.  To take advantage of this structure, the ARKode package calls a custom HOMME-NH routine to solve for $\delta_{m,j}^{(k+1)}$ from $E_{m,j}$, $\Delta t$, $\hat{A}_{j,j}$, and $g_{m,j}^{(k)}$.  In this custom routine, components of $\delta_{m,j}^{(k+1)}$ not pertaining to $g_{m,j}^w$ or $g_{m,j}^\phi$ are set to the values of the corresponding components of $E_{m,j}$.  Components of $\delta_{m,j}^{(k+1)}$ pertaining to $g_{m,j}^\phi$, denoted $\delta_{m,j}^{\phi,(k+1)}$, are computed by decomposing the linear system $J_{m,j}(g_{m,j}^{(k)}) \delta_{m,j}^{\phi,(k+1)} = E_{m,j}^\phi$ into the independent tridiagonal blocks for each grid column.  The LAPACK routines DGTTRF and DGTTRS are called to solve for $\delta_{m,j}^{\phi,(k+1)}$, which is then used to complete $\delta_{m,j}^{(k+1)}$ via \eqref{eq:wphirelation}: $\delta_{m,j}^{(k+1)w} = (\delta_{m,j}^{(k+1)\phi} - E_{m,j}^\phi)/(\mathfrak{g}\Delta t \hat{A}_{j,j})$.

The ARKode package generates iterates $\delta_{m,j}^{(k+1)}$ until $R_{m,j}^{(k+1)}\|\delta_{m,j}^{(k+1)}\| < \epsilon$, where
\begin{equation*}
  R_{m,j}^{(k+1)} = \max \left ( 0.3 R_{m,j}^{(k)}, \frac{\|\delta_{m,j}^{(k+1)}\|}{\|\delta_{m,j}^{(k)}\|} \right ), \,
  \|\delta_{m,j}^{(\cdot)}\| = \left [ \frac{1}{N} \sum_{l=1}^N \left( \frac{[\delta_{m,j}^{(\cdot)}]_l}{\epsilon_r |[x_{m,j}]_l| + [\epsilon_a]_l} \right)^2 \right]^{\frac{1}{2}},
\end{equation*}
$R_{m,i}^{(0)} = 1$, $N$ is the total number of components in $q_m$, and $[\cdot]_l$ indicates selecting the $l^\text{th}$ element.  Note that $\epsilon$, $\epsilon_r$, and $\epsilon_a$ are all tunable tolerances.  The value of $\epsilon$ chosen here is the default ARKode value $\epsilon = 0.1$.  We chose $\epsilon_r = 10^{-6}$ by varying the value until the change in solution was negligible.  For the absolute tolerances, we chose $\epsilon_a^u = \epsilon_a^v = \epsilon_a^w = 10 \epsilon_r$, $\epsilon_a^\phi = 10^5\epsilon_r$, $\epsilon_a^\Theta = 10^6\epsilon_r$, and $\epsilon_a^{\partial \pi / \partial \eta} = \epsilon_r$.  Those coefficients correspond to the general expected magnitude of each of the quantities.

\subsection{DCMIP Test Cases and small planets}\label{sec:dcmip}

We use two test cases from the 2012 Dynamical Core Model Intercomparison Project (DCMIP2012) \cite{dcmip2012}: the nonhydrostatic gravity wave test case (DCMIP2012.3.1) and the dry baroclinic instability test case (DCMIP2012.4.1).  These test cases make use of planets whose radiuses can vary while atmospheric depth and gravity are held constant.  For example, small planet $\times 100$ is a planet whose radius is 1/100 that of the Earth's while its atmospheric depth and gravity are the same as the Earth's.  The amount by which the planet is scaled is referred to as the planet size.  The DCMIP2012.3.1 test case is run with small planet $\times 125$ and the DCMIP2012.4.1 test case is run with planet sizes $1$, $10$, and $100$. 

Small planets enable simulation of various vertical-to-horizontal aspect ratios without computationally expensive experiments at high horizontal resolution.  Determining the maximum usable step-size of many methods at various aspect ratios can then be done in a reasonable amount of time.  A small planet $\times n$ simulation is comparable in terms of stability to a normal size planet simulation where the vertical-to-horizontal aspect ratio is scaled by $n$.  For explicit methods running HOMME-NH in hydrostatic mode (KGU35(H) in Table \ref{tab:maxusablestep}), the maximum usable step-size scales as the planet size.  This is not necessarily true for IMEX RK methods (see Table \ref{tab:maxusablestep}).

\subsection{Test Results}\label{sec:testresults}

\subsubsection{Accuracy}\label{sec:accuracy}

We present the results of a formal convergence study of the best performing (in terms of the results of Section \ref{sec:imkgstabilityresults}) IMKG2-3 methods.  These methods are used for integration of the DCMIP2012.3.1 nonhydrostatic gravity wave test case with small planet $\times 125$ and $ne=27$ cubed sphere resolution with 20 vertical levels.  We generate an approximate reference solution over a 5 hour window using the explicit KGU35 method with the (very small) step-size of $\Delta t = 3.9 \cdot 10^{-4}$.

As discussed in Section \ref{sec:impdetails}, we run without hyperviscosity to avoid a reduction to first order temporal accuracy.  Without artificial damping of the high frequency modes, large and unphysical oscillations generated by the spectral element discretization can destabilize longer simulations.  Thus, we restrict the simulations to 5 hours when running without hyperviscosity.  The results in Figure \ref{fig:convergence} show that the IMKG2-3 methods we test attain their formal convergence order until the error reaches that of machine round-off accumulation.

\begin{figure}
\centering
\includegraphics[scale=0.25]{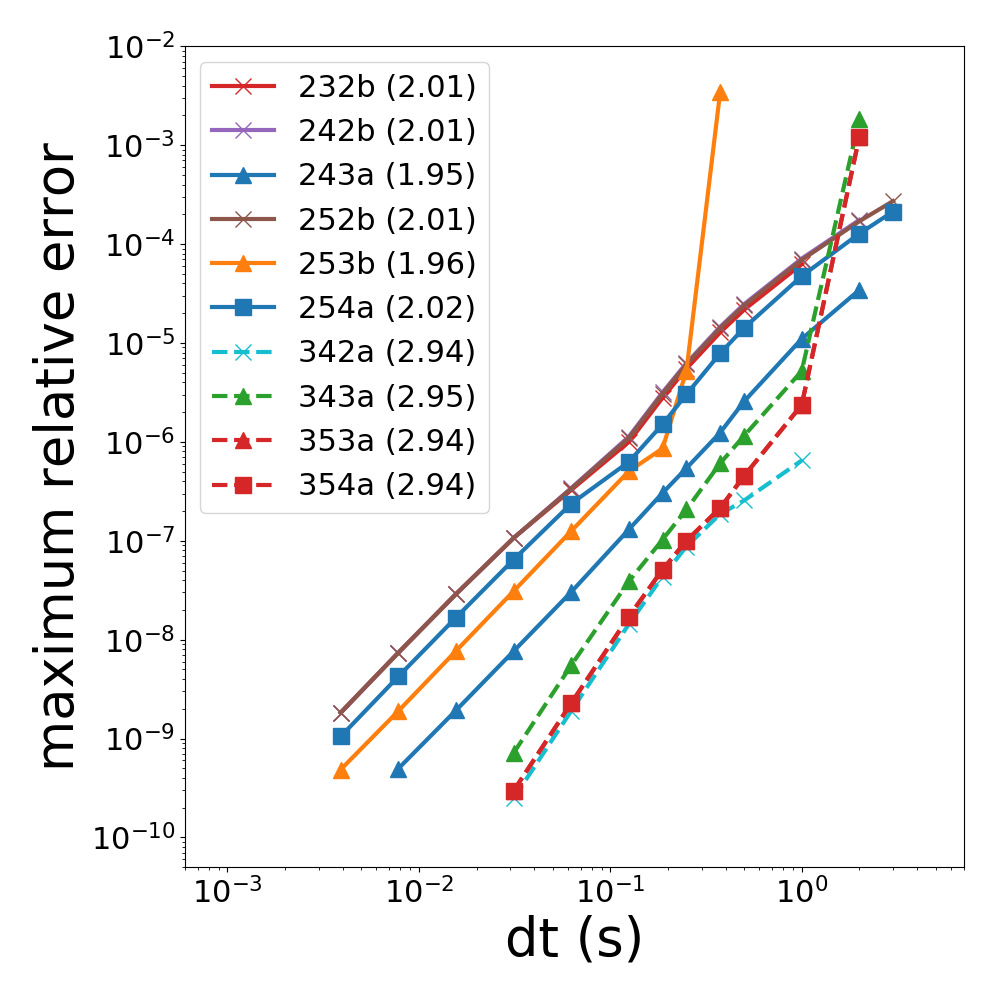}
\caption{Plot of the maximum relative error in the temperature field vs step-size of various IMKG2-3 methods for the DCMIP2012.3.1 test after a 5 hour run.  The value in parentheses next to the method name represents the best approximation to the order of convergence.}
\label{fig:convergence}
\end{figure}

\subsubsection{Stability and efficiency at the maximum usable step-size}\label{sec:imkgstabilityresults}

We present results for the maximum usable step-size (MUS) of various IMKG2-3 and non-IMKG methods (Tables \ref{tab:maxusablestep}-\ref{tab:score}).  We also present results on the total time to solution or run-time when these methods are run at their MUS (Figures \ref{fig:runtime1}-\ref{fig:runtime2}).  All runs use the DCMIP2012.4.1 test case on the $ne=30$ cubed sphere grid with 30 vertical levels and planet sizes $1$, $10$, and $100$.  The DCMIP2012.4.1 test case is employed since this test case generates the types of nontrivial flow expected in production runs.

Although running methods at or near their MUS risks producing an inaccurate solution, this is common practice for global atmosphere models as noted in \cite{temporalerror}.  The results we present are still useful since they can give estimates of the computational cost per time-step, the minimal time to solution, and the computational scaling of various IMEX RK methods relative to each other.  See \cite{voglpaper2019} for evaluation of the accuracy of IMKG methods run with large time-steps in HOMME-NH and \cite{hommexx} for analysis of its computational scaling.


The MUS is determined empirically: for simulation of planet size $n$, we attempt runs with a time-step of $25k/n$ starting from $k=1$, then $k=2$, and so forth.  If $k_0$ is the minimal $k$ such that a run with time-step $25k/n$ fails to complete due to solution blow-up or a solver failure, then we call $25(k_0-1)/n$ the MUS.  Run-time experiments (Figures \ref{fig:runtime1}-\ref{fig:runtime2}) were conducted on a local computing cluster using 5 or 150 dual socket nodes.  Each socket contains 18, 2.1 GHz, Intel Broadwell E5-2695 v4 computing cores.  A cubed sphere with the $ne = 30$ horizontal resolution has $5400$ horizontal elements, with one vertical column per element, so that 5 and 150 node runs correspond to ratios of 30 horizontal elements per computing core (elements/core) and 1 element/core, respectively.  These runs represent extremes of computational scaling from the 30 element/core regime where parallel communication is cheap to the 1 element/core regime where parallel communication is expensive.

We consider methods to be efficient if they have a large MUS relative to the number of explicit function evaluations and implicit solves they require and at a variety of vertical-to-horizontal aspect ratios.  From run-time data, we observed that the cost of an implicit solve relative to an explicit function evaluation varied between about 55-60\% for the 30 element/core runs down to about 35-40\% for the 1 element/core runs.  In Table \ref{tab:score} we scale the MUS of various methods (Table \ref{tab:maxusablestep}) by their number of explicit function evaluations and implicit stages assuming that an implicit solve costs 50\% of an explicit function evaluation. While this metric assumes that the cost of an implicit solve does not vary between methods, machines, or implementations, it still gives a coarse measure of run-time performance (see Figures \ref{fig:runtime1}-\ref{fig:runtime2}).


We now discuss run-time performance for IMKG2-3 methods and the non-IMKG methods (Figures \ref{fig:runtime1}-\ref{fig:runtime2}).  The run-time of a method depends on the planet-size as well as the number of elements/core used in the simulation.  For several methods (e.g. IMGK253b and ARK346), the run-times are much longer for small planet $\times 100$ compared to small planet $\times 1$, reflecting the results of Table \ref{tab:maxusablestep}.   This means that for a fixed resolution and fixed number of elements/core, the relative performance of two methods can change as the planet-size is decreased (e.g. the IMKG252b and IMKG253b methods with either 1 or 30 elements/core).  Similarly, the relative performance of methods can change as the number of elements/core decreases (e.g. ARK2 and ARS232 for small planet $\times 100$ or IMKG232b and IMKG343a for small planet $\times 10$).

At 30 elements/core, the only non-IMKG methods whose run-time performance is competitive with any IMKG2-3 methods are ARK324 for small planet $\times 1$ and the ARK2 method for small planet $\times 100$.  The IMKG242b and IMKG252a methods outperform every non-IMKG method for every planet size and the remaining tested IMKG2-3 methods are outperformed only in some cases by the ARK324 method for small planet $\times 1$ and the ARK2 method for small planet $\times 1$.  

At 1 element/core,  IMKG252b has better run-time performance than any of the non-IMKG methods at every planet-size.  IMKG242b outperforms every non-IMKG method except ARS343 for small planet $\times 1$.  ARK324 and ARS343 perform well for small planet $\times 1$ while ARK2 and ARS232 perform well for small planet $\times 100$; however, many of the IMKG2-3 methods (IMKG242b, IMKG252b, IMKG243a, IMKG253b, IMKG254a-b, IMKG353a, IMKG354a) complete 15 day runs in less than 46 seconds for every tested planet size.   This run-time performance is not matched for any of the tested non-IMKG methods.  Several IMKG2-3 methods perform well for every tested planet size, whereas this is not true of the non-IMKG methods.  This reflects the derivation of the tested IMKG2-3 methods that relied on analysis of their H-stability regions (Section \ref{sec:imkgstability}).

Several methods are capable of running in nonhydrostatic mode at or near the hydrostatic step-size; their predicted step-size when run in hydrostatic mode.  For example, the MUS of the IMKG254a method in nonhydrostatic mode is equal to or just less than the MUS of the KGU35 explicit RK method in hydrostasic mode for planet sizes $\times 1$, $\times 10$, and $\times 100$.  Since KGU35 and IMKG254a are five stage methods, KGU35 has a KGNO stability polynomial, and the explicit method of IMKG254a has a KGO stability polynomial, it follows that the MUS of IMKG254a is at or near the hydrostatic step-size for planet sizes $\times 1$, $\times 10$, and $\times 100$.  When adjusting for the fact that IMKG232b and IMKG243a have three and four explicit stages respectively, it follows that the MUS of these two methods is also at or near their expected hydrostatic step-size for planet sizes $\times 1$, $\times 10$, and $\times 100$.  None of the non-IMKG methods tested are capable of running at or near a hydrostatic step-size for all tested planet sizes.

\begin{table}[htb]
    \label{tab:maxusablestep}
\small
\centering
\tabcolsep=.11cm
    \begin{tabular}{ l|llllllllllllllll}

IMKG    & 232a & 232b & 242a & 242b & 243a & 252a & 252b & 253a  \\ \hline
 x1     & 100  & 200  & 175  & 225  & 275  & 150  & 275  & 200   \\
 x10    & 10   & 17.5 & 17.5 & 27.5 & 27.5 & 15   & 37.5 & 20  \\ 
 x100   & 1.75 & 1.75 & 2.25 & 2.5  & 2.5  & 2.5  & 3.5  & 2.25  \\ \hline

     \end{tabular}
     
         \begin{tabular}{ l|lllllllllllllllll}

IMKG  & 253b & 254a & 254b & 254c  & 342a & 343a & 353a & 354a \\ \hline
 x1   & 375  & 375  & 375  & 150   & 75 & 275 & 250 & 350 \\
 x10  & 32.5 & 37.5 & 35   & 15    & 22.5 & 22.5 & 25 & 32.5 \\ 
 x10  & 2.5  & 3.5  & 3.0  & 2.25  & 2.25 & 2.25 & 2.5 & 2.75\\ \hline

     \end{tabular}

  \tabcolsep=0.11cm
    \begin{tabular}{ l|lllllllllllll}   
     Method & KGU35 & KGU35(H) & ARS232 & ARS343 & ARS443 & ARK2 & ARK324 & ARK346 \\ \hline
 x1    & 0.75 & 375   &  125   & 275    & 175  & 125 & 250 & 275 \\ 
x10    & 0.75 & 37.5  & 12.5   & 17.5   & 15   & 12.5   & 17.5 & 20 \\
x100   & 0.75 & 3.75  & 1.75 & 1.5 & 1.75 & 1.75 & 1.75 & 1.5 \\

          \end{tabular}
    \caption{MUS for various IMEX RK methods and the KGU35 method running in both in nonhydrostatic and hydrostatic (KGU35(H)) modes  with $ne = 30$, 30 vertical levels, and small planet $\times 1$, $\times 10$, or $\times 100$.   }

\end{table}

\begin{table}[htb]
    \label{tab:score}
\small
\centering
\tabcolsep=.11cm
    \begin{tabular}{ l|lllllllllllllllll}

IMKG    & 232a & 232b & 242a & 242b & 243a & 252a & 252b & 253a  \\ \hline
 x1     & 25   & 50   & 35  & 45  & 50   & 25    & 45.8 & 30.8   \\
 x10    & 2.5  & 4.38 & 3.5 & 5.5 & 5    & 2.5   & 6.25 & 3.08  \\ 
 x100   & .438 & .438 & .45 & .50 & .455 & .416  & .583 & .346  \\ \hline

     \end{tabular}
     
         \begin{tabular}{ l|llllllllllllllllll}

IMKG  & 253b  & 254a & 254b  & 254c & 342a & 343a & 353a  & 354a \\ \hline
 x1   & 57.7  & 53.6 & 53.6  & 21.4 & 15   & 50   & 38.5  & 50 \\
 x10  & 5     & 5.36 & 5     & 2.14 & 4.5  & 4.09 & 3.85  & 4.64 \\ 
 x10  & .385  & 0.5  & .429  & .321 & .45  & .409 & .3853 & .393 \\ \hline

     \end{tabular}

  \tabcolsep=0.11cm
    \begin{tabular}{ l|lllllllllllll}   
     Method & KGU35 & KGU35(H) & ARS232 & ARS343 & ARS443 & ARK2 & ARK324 & ARK346 \\ \hline
 x1    & 0.15 & 75   &  41.7 & 61.1 & 29.2 & 41.7 & 55.6 & 36.7 \\ 
x10    & 0.15 & 7.5  & 4.17  & 3.89 &  2.5 & 4.17 & 3.89 & 2.67 \\
x100   & 0.15 & .75  & .583  & .389 & .292 & .583 & .389 & .2 \\

          \end{tabular}
    \caption{MUS from Table \ref{tab:maxusablestep} scaled by the number of required explicit function evaluations and implicit stages assuming that an implicit solve costs 50\% of an explicit function evaluation. }

\end{table}

\begin{figure}[htb]
\includegraphics[scale=0.4]{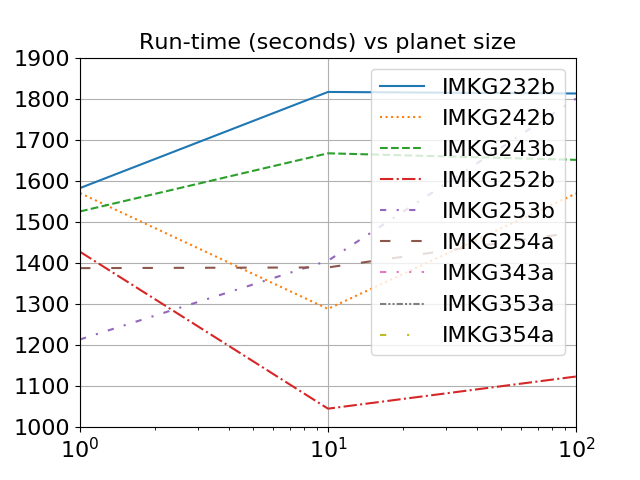}\includegraphics[scale=0.4]{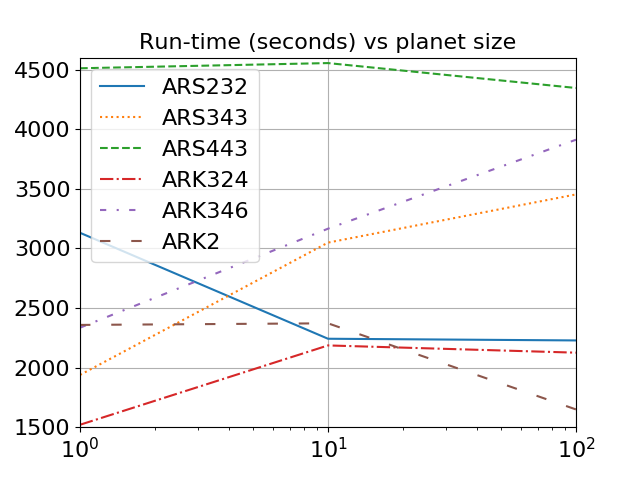}
\caption{Run-time vs planet size of several IMKG2-3 and non-IMKG methods running at their maximum usable time-step with 30 elements/core, $ne=30$, and 30 vertical levels. }
\label{fig:runtime1}
\end{figure}

\begin{figure}[htb]
\includegraphics[scale=0.4]{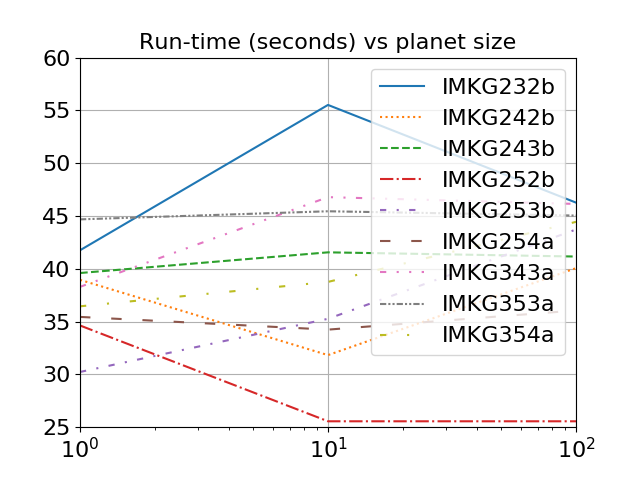}\includegraphics[scale=0.4]{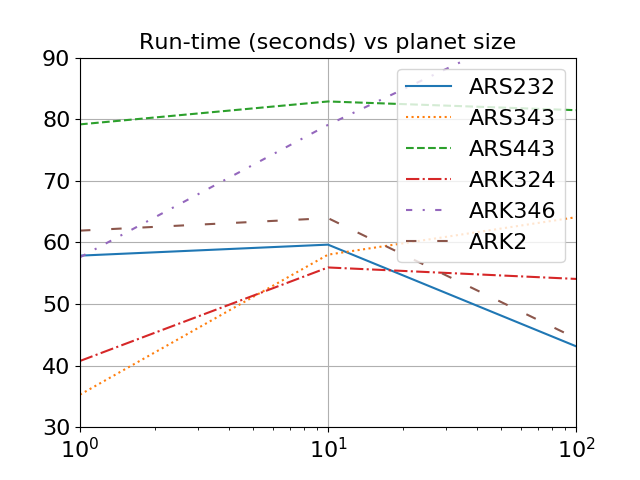}
\caption{Run-time vs planet size of several IMKG2-3 and non-IMKG methods running at their maximum usable time-step with 1 element/core, $ne=30$, and 30 vertical levels. }
\label{fig:runtime2}
\end{figure}


\section{Conclusion and Acknowledgements}\label{sec:conclusions}

In this paper we have analyzed a new family of second and third order accurate IMEX Runge-Kutta methods for nonhydrostatic atmosphere models.  H-stability is used to derive methods capable of running with large stable step-sizes, including several methods capable of running with a hydrostatic step-size, to integrate a nonhydrostatic atmosphere model with a HEVI partitioning.  The analysis presented in this paper can be readily modified to derive second and third order accurate IMEX RK methods for applications with different stability requirements.

We acknowledge David Gardner, Professor Dan Reynolds, and Carol Woodward from the SUNDIALS-ARKode development team for their help in developing and implementing the ARKode-HOMME-NH interface and their advice regarding IMEX methods and solvers.  We also thank Professor Paul Ullrich for his expertise and advice on structuring the paper that led to an improved manuscript.



\section{Appendix}\label{sec:appendix}

 Each IMKG2-3 method is named IMKG$p f j l$ where $p$ is the order of accuracy, $f$ is the number of explicit stages, $j$ is the number of implicit stages, and $l$ is an identifying letter.  We express IMKG2-3 methods using five vectors (Section \ref{sec:methods}): $\alpha$, $\beta$, $\hat\alpha$, $\hat\beta$, and $\hat\delta$  (we omit $\beta$ for IMKG2 methods since it is always zero for them).  Method coefficients for IMKG2 and IMKG3 methods are given in Tables \ref{tab:imkg2methods} and \ref{tab:imkg3methods} respectively.  Various properties of these methods are displayed in Table \ref{tab:imkg_properties}.

\begin{table}[htb]
\footnotesize
\centering
\begin{tabular}{ c|ccc }

  & & &\\[-1em]
 IMKG  & $\alpha$ & $\hat{\alpha}$ & $\hat{\delta}$ \\ \hline
   & & &\\[-1em]
232a & $\left(\frac{1}{2},\frac{1}{2},1\right)^T$  & $\left(0,0,\frac{\sqrt{2}-1}{2}\right)^T$ & $\left(\frac{2-\sqrt{2}}{2},\frac{2-\sqrt{2}}{2}\right)^T$  \\[.4em]
 & & &\\[-1em]
232b & $\left(\frac{1}{2},\frac{1}{2},1\right)^T$ &\ $\left(0,0,-\frac{1+\sqrt{2}}{2} \right)^T$ & $\left(0, \frac{2+\sqrt{2}}{2},\frac{2+\sqrt{2}}{2}\right)^T$ \\[.4em]
 & & &\\[-1em]
242a & $\left(\frac{1}{4},\frac{1}{3},\frac{1}{2},1 \right)^T$ & $\left(0,0,\frac{\sqrt{2}-1}{2},1\right)^T$ & $\left(0,0,\frac{2-\sqrt{2}}{2},\frac{2-\sqrt{2}}{2} \right)^T$\\[.4em]
 & & &\\[-1em]
242b & $\left(\frac{1}{4},\frac{1}{3},\frac{1}{2},1 \right)^T$ & $\left(0,0,-\frac{1+\sqrt{2}}{2},1 \right)^T$ & $\left(0,0,\frac{2+\sqrt{2}}{2},\frac{2+\sqrt{2}}{2} \right)^T$\\[.4em]
 & & &\\[-1em]
243a & $\left(\frac{1}{4},\frac{1}{3},\frac{1}{2},1 \right)^T$ & $\left(0,\frac{1}{6},\frac{\sqrt{3}}{6},1 \right)^T$ & $\left(0,\frac{1}{2}+\frac{\sqrt{3}}{6} ,\frac{1}{2}+\frac{\sqrt{3}}{6} ,\frac{1}{2}+\frac{\sqrt{3}}{6} \right)^T$\\[.4em]
 & & &\\[-1em]
252a & $\left(\frac{1}{4},\frac{1}{6},\frac{3}{8},\frac{1}{2},1 \right)^T$ & $\left(0,0,\frac{\sqrt{2}-1}{2},1\right)^T$ &   $\left(0,0,0,\frac{2-\sqrt{2}}{2},\frac{2\sqrt{2}}{2} \right)^T$\\[.4em]
 & & &\\[-1em]
 252b & $\left(\frac{1}{4},\frac{1}{6},\frac{3}{8},\frac{1}{2},1 \right)^T$ & $\left(0,0,-\frac{1+\sqrt{2}}{2},1\right)^T$ &   $\left(0,0,0,\frac{2+\sqrt{2}}{2},\frac{2+\sqrt{2}}{2} \right)^T$\\[.4em]
 & & &\\[-1em]
253a & $\left(\frac{1}{4},\frac{1}{6},\frac{3}{8},\frac{1}{2},1 \right)^T$ & $\left(0,\gamma_-,\frac{\sqrt{3}}{6},1\right)^T$ & $\left(0, \frac{1}{2}-\frac{\sqrt{3}}{6}, \frac{1}{2}-\frac{\sqrt{3}}{6}, \frac{1}{2}-\frac{\sqrt{3}}{6} \right)^T$\\[.4em]
 & & &\\[-1em]
 253b & $\left(\frac{1}{4},\frac{1}{6},\frac{3}{8},\frac{1}{2},1 \right)^T$ & $\left(0,\gamma_+,-\frac{\sqrt{3}}{6},1\right)^T$ & $\left(0, \frac{1}{2}+\frac{\sqrt{3}}{6}, \frac{1}{2}+\frac{\sqrt{3}}{6}, \frac{1}{2}+\frac{\sqrt{3}}{6} \right)^T$\\[.4em]
 & & &\\[-1em]
 & & &\\[-1em]
254a  & $\left(\frac{1}{4},\frac{1}{6},\frac{3}{8},\frac{1}{2},1 \right)^T$ & $\left(0,-\frac{3}{10} ,\frac{5}{6},-\frac{3}{2}\right)^T$ & $\left(-\frac{1}{2},1,1,2 \right)^T$\\[.4em]
 & & &\\[-1em]
 254b & $\left(\frac{1}{4},\frac{1}{6},\frac{3}{8},\frac{1}{2},1 \right)^T$ & $\left(0,-\frac{1}{20} ,\frac{5}{4},-\frac{1}{2}\right)^T$ & $\left(-\frac{1}{2},1,1,1 \right)^T$\\[.4em]
 & & &\\[-1em]
 
254c & $\left(\frac{1}{4},\frac{1}{6},\frac{3}{8},\frac{1}{2},1 \right)^T$ & $\left(0,\frac{1}{20},\frac{5}{36},\frac{1}{3},1\right)^T$ & $\left(\frac{1}{6},\frac{1}{6},\frac{1}{6},\frac{1}{6}\right)^T$      \\[.4em]
\end{tabular}
\caption{Method coefficients of IMKG2 methods where $\gamma- = 0.08931639747704086$ and $\gamma_+ = 1.2440169358562922$.}
\label{tab:imkg2methods}

\footnotesize
\centering
\begin{tabular}{ c|cccc } 

  & & \\[-1em]
 IMKG  & $\alpha$ & $\hat{\alpha}$  \\ \hline
 & & \\[-1em]
 342a & $\left(\frac{1}{3}, \frac{1}{3},\frac{3}{4}\right)^T$ & $\left(0,-\frac{1+\sqrt{3}}{6},-\frac{1+\sqrt{3}}{6},\frac{3}{4}\right)^T$  \\[.4em]
 & & \\[-1em]
343a & $\left(\frac{1}{4},\frac{2}{3},\frac{1}{3},\frac{3}{4} \right)^T$ & $\left(0,-\frac{1}{3},-\frac{2}{3},\frac{3}{4}\right)^T$  \\[.4em]
 & & \\[-1em]
353a & $\left(\frac{1}{4},\frac{2}{3},\frac{1}{3},\frac{3}{4} \right)^T$ & $\left( 0,-\frac{359}{600},-\frac{559}{600},\frac{3}{4}\right)^T$ \\[.4em]
 & & \\[-1em]
354a & $\left(\frac{1}{5},\frac{1}{5},\frac{2}{3},\frac{1}{3},\frac{3}{4} \right)^T$ & $\left(0,0,\frac{11}{30},-\frac{2}{3},\frac{3}{4}\right)^T$ \\[.4em]\hline
 & & \\[-1em]
 & $\hat{\delta}$ & $\beta$ \\ 
  & & \\[-1em]
342a &  $\left( 0,\frac{1+\sqrt{3}/3}{2},\frac{1+\sqrt{3}/3}{2}\right)^T$ & $(\frac{1}{3},\frac{1}{3},\frac{1}{4})^T$ \\[.4em]
343a &  $(-\frac{1}{3},1,1)^T$ & $(0,\frac{1}{3},\frac{1}{4})^T$\\[.4em]
 & & \\[-1em]
353a &  $\left(-1.1678009811335388,1.265,1.265\right)^T$ & $(0,0,\frac{1}{3},\frac{1}{4})^T$\\[.4em]
 & & \\[-1em]
354a &  $\left( 0,\frac{2}{5},\frac{2}{5},1\right)^T$ & $(0,0,\frac{1}{3},\frac{1}{4})^T$\\[.4em]
\end{tabular}
\caption{Method coefficients IMKG3 methods.}
\label{tab:imkg3methods}
\end{table}

\begin{table}[htb]
\footnotesize
\centering
\begin{tabular}{ c|ccccccccccccccc }

IMKG & 232a & 232b & 242a & 242b & 243a & 243b & 252a & 252b & 253a & 253b  \\\hline
    & & &\\[-1em]
 I or A & A & A & A & A & A & A & A & A & A & A \\
    & & &\\[-1em]
 VI & Y & Y & N & Y & Y & Y  & N &  N & Y & Y \\
    & & &\\[-1em]
 SD & Y & Y & Y & Y & Y & Y & Y & Y & Y & Y   \\\hline 
 \end{tabular}
 \begin{tabular}{c|ccccccc}
IMKG & 254a & 254b & 254c & 342a & 343a & 353a & 354a \\\hline
    & & &\\[-1em]
 I or A & I & I & A & A & I & A & I \\ 
    & & &\\[-1em]
 VI & Y & Y & Y & N & Y & Y & Y \\ 
    & & &\\[-1em]
 SD & N & N & Y & Y & N & Y & N \\ 
\end{tabular}
\caption{Properties of IMKG2-3 methods used in Section \ref{sec:experiments} with double Butcher tableaux defined in Tables \ref{tab:imkg2methods} and \ref{tab:imkg3methods}: if the implicit method is I- or A-stable (I or A), if the implicit method is VI (YES(Y) or NO(N)), and if the implicit method is SD (YES(Y) or NO(N)).}
\label{tab:imkg_properties}

\end{table}

\clearpage

\bibliographystyle{siamplain}

\begin{thebibliography}{10}

\bibitem{accerlatenuma}
{\sc {Abdi, D.}, {Giraldo, F.}, {Constantinescu, M.}, {Carr III, L.}, {Wilcox,
  L.}, and {Warburton, T.}}, {\em Acceleration of the {IM}plicit-{EX}plicit
  non-hydrostatic unified model of the atmosphere ({NUMA}) on manycore
  processors}, Int. J. High Perform C., 33 (2019),
  \url{https://doi.org/10.1177/1094342017732395}.

\bibitem{ARS1997}
{\sc {Ascher, U.}, {Ruuth, S.}, and {Spiteri, R.}}, {\em Implicit-explicit
  {R}unge-{K}utta methods for time-dependent partial differential equations},
  Appl. Numer. Math., 25 (1997), pp.~151--167,
  \url{https://doi.org/10.1137/0732037}.

\bibitem{hommexx}
{\sc {Bertagna, L.}, {Deakin, M.}, {Guba, O.}, {Sunderland, D}, {Bradley, A.},
  {Tezaur, I.}, {Taylor, M.}, and {Salinger, A.}}, {\em {HOMMEXX} 1.0: {A}
  performance portable atmospheric dynamical core for the energy exascale earth
  system model}, Geosci. Model Dev., 12 (2019), pp.~1423--1441,
  \url{https://doi.org/10.5194/gmd-12-1423-2019}.

\bibitem{camse2012}
{\sc {Dennis, J.}, {Edwards, J.}, {Evans, K.}, {Guba, O.}, {Lauritzen, P.},
  {Mirin, A.}, {St-Cyr, A.}, {Taylor, M.}, and {Worley, P.}}, {\em {CAM-SE}: A
  scalable spectral element dynamical core for the {C}ommunity {A}tmosphere
  {M}odel}, Int. J. High Perform C., 26 (2012), pp.~74--89,
  \url{https://doi.org/10.1177/1094342011428142}.

\bibitem{BD2012}
{\sc {Durran, D.} and {Blossey, P.}}, {\em Implicit-explicit multistep methods
  for fast-wave-slow-wave problems}, Mon. Weather Rev., 140 (2012),
  pp.~1307--1325, \url{https://doi.org/10.1175/MWR-D-11-00088.1}.

\bibitem{cam42013}
{\sc {Evans, K.}, {Lauritzen, P.}, {Mishra, S.}, {Neale, R.}, {Taylor, M.}, and
  {Tribbia, J.}}, {\em {AMIP} simulation with the {CAM4} spectral element
  dynamical core}, J. Climate, 26 (2013), pp.~689--709,
  \url{https://doi.org/10.1175/JCLI-D-11-00448.1}.

\bibitem{JHV1997}
{\sc {Frank, J.}, {Hundsdorfer, W.}, and {Verwer, J.}}, {\em On the stability
  of implicit-explicit linear multistep methods}, Appl. Numer. Math., 25
  (1997), pp.~193--205, \url{https://doi.org/10.1016/S0168-9274(97)00059-7}.

\bibitem{ARKODE2018}
{\sc {Gardner, D.}, {Guerra, J.}, {Hamon, F.}, {Reynolds, D.}, {Ullrich, P.},
  and {Woodward, C.}}, {\em Implicit-explicit ({IMEX}) {R}unge-{K}utta methods
  for non-hydrostatic atmospheric models}, Geosci. Model Dev., 11 (2018),
  pp.~1497--1515, \url{https://doi.org/10.5194/gmd-2017-285}.

\bibitem{Gardner2017}
{\sc {Gardner, D.}, {Reynolds, D.}, {Hamon, F.}, {Woodward, C.}, {Ullrich, P.},
  {Guerra, J.}, {Lelbach, B.}, and {Banide, A.}}, {\em Tempest+{ARK}ode {IMEX}
  tests}, Geosci. Model Dev.,  (2017),
  \url{https://doi.org/10.5281/zenodo.1162309}.

\bibitem{GC2016}
{\sc {Ghosh, D.} and {Constantinescu, E.}}, {\em Semi-implicit time integration
  of atmospheric flows with characteristic-based flux partitioning}, SIAM J.
  Sci. Comput., 38 (2016), pp.~A1848--A1875,
  \url{https://doi.org/10.1137/15M1044369}.

\bibitem{GKC2013}
{\sc {Giraldo, F.}, {Kelly, J.}, and {Constantinescu, E.}}, {\em
  Implicit-explicit formulations of a three-dimensional nonhydrostatic unified
  model of the atmosphere ({NUMA})}, SIAM J. Sci. Comput., 35 (2013),
  pp.~B1162--B1194, \url{https://doi.org/10.1137/120876034}.

\bibitem{GRL2010}
{\sc {Giraldo, F.}, {Rastelli, M.}, and {L\"{a}uter, M.}}, {\em Semi-implicit
  formulations of the {N}avier-{S}tokes equations: Application to
  nonhydrostatic atmospheric modeling}, SIAM J. Sci. Comput., 32 (2010),
  pp.~3394--3425, \url{https://doi.org/10.1137/090775889}.

\bibitem{guerraullrich2016}
{\sc {Guerra, J.} and {Ullrich, P.}}, {\em A high-order staggered
  finite-element vertical discretization for non-hydrostatic atmospheric
  models}, Geosci. Model Dev., 9 (2016), pp.~2007--2029,
  \url{https://doi.org/10.5194/gmd-9-2007-2016}.

\bibitem{Hairer1981}
{\sc {Hairer, E.}}, {\em Order conditions for numerical methods for partitioned
  ordinary differential equations}, Numer. Math., 36 (1981), pp.~431--445,
  \url{https://doi.org/10.1007/BF01395956}.

\bibitem{BHL1982}
{\sc {Hairer, E.}, {Bader, G.}, and {Lubich C.}}, {\em On the stability of
  semi-implicit methods for ordinary differential equations}, BIT, 22 (1982),
  pp.~211--232, \url{https://doi.org/10.1007/BF01944478}.

\bibitem{hindmarsh2005sundials}
{\sc {Hindmarsh, A.}, {Brown, P.}, {Grant, K.}, {Lee, S.}, {Serban, R.},
  {Shumaker, D.}, and {Woodward, C.}}, {\em {SUNDIALS}: Suite of nonlinear and
  differential/algebraic equation solvers}, ACM T. Math. Software, 31 (2005),
  pp.~363--396.

\bibitem{JI2017}
{\sc {Izzo, G.} and {Jackiewicz, Z.}}, {\em Highly stable implicit-explicit
  {R}unge-{K}utta methods}, Appl. Numer. Math., 113 (2017), pp.~71--92,
  \url{https://doi.org/10.1016/j.apnum.2016.10.018}.

\bibitem{JV2000}
{\sc {Jackiewicz, Z.} and {Vermiglio, R.}}, {\em Order conditions for
  partitioned {Runge-Kutta methods}}, Appl. Math.-Czech, 45 (2000),
  pp.~301--316, \url{https://doi.org/10.1023/A:1022323529349}.

\bibitem{JN1981}
{\sc {Jeltsch, R.} and {Nevanlinna, O.}}, {\em Stability of explicit time
  discretizations for solving initial value problems}, Numer. Math., 37 (1981),
  pp.~61--91, \url{https://doi.org/10.1007/BF01396187}.

\bibitem{temporalerror}
{\sc {Jia, J.}, {Hill, J.}, {Evans, K.}, {Fann, G.}, and {Taylor, M.}}, {\em A
  spectral deferred correction method applied to the shallow water equations on
  a sphere}, Mon. Weather Rev., 141 (2013), pp.~3435--3449,
  \url{https://doi.org/10.1175/MWR-D-12-00048.1}.

\bibitem{Kasahara1974}
{\sc {Kasahara, A.}}, {\em Various vertical coordinate systems used for
  numerical weather prediction}, Mon. Weather Rev., 102 (1974), pp.~509--522,
  \url{https://doi.org/10.1175/1520-0493(1974)102<0509:VVCSUF>2.0.CO;2}.

\bibitem{CK2003}
{\sc {Kennedy, C.} and {Carpenter, M.}}, {\em Additive {R}unge-{K}utta schemes
  for convection-diffusion-reaction equations}, Appl. Numer. Math., 44 (2003),
  pp.~139--181, \url{https://doi.org/10.1016/S0168-9274(02)00138-1}.

\bibitem{KG1984a}
{\sc {Kinnmark, I.} and {Gray, W.}}, {\em One step integration methods with
  maximum stability regions}, Math. Comput. Simulat., XXVI (1984), pp.~84--92,
  \url{https://doi.org/10.1016/0378-4754(84)90039-9}.

\bibitem{KG1984b}
{\sc {Kinnmark, I.} and {Gray, W.}}, {\em One step integration methods with
  third-fourth order accuracy with large hyperbolic stability limits}, Math.
  Comput. Simulat., XXVI (1984), pp.~181--188,
  \url{https://doi.org/10.1016/0378-4754(84)90056-9}.

\bibitem{Laprise1992}
{\sc {Laprise, R.}}, {\em The {E}uler equations of motion with hydrostatic
  pressure as an independent variable}, Mon. Weather Rev., 102 (1992),
  pp.~197--207,
  \url{https://doi.org/10.1175/1520-0493(1992)120<0197:TEEOMW>2.0.CO;2}.

\bibitem{LWW2014}
{\sc {Lock, S.-J.}, {Wood, N.}, and {Weller, H.}}, {\em Numerical analyses of
  {R}unge-{K}utta implicit-explicit schemes for horizontally explicit,
  vertically implicit solutions of atmospheric models}, Q. J. Roy. Meteor.
  Soc., 140 (2014), pp.~1654--1669, \url{https://doi.org/10.1002/qj.2246}.

\bibitem{Lorenz1960}
{\sc {Lorenz, E.}}, {\em Energy and numerical weather prediction}, Tellus, 12
  (1960), pp.~364--373,
  \url{https://doi.org/10.1111/j.2153-3490.1960.tb01323.x}.

\bibitem{RSSZ2017}
{\sc {Rosales, R.}, {Seibold, B.}, {Shirokoff, D.}, and {Zhou, D.}}, {\em
  Unconditional stability for multistep {I}m{E}x schemes: Theory}, SIAM J.
  Numer. Anal., 55 (2017), pp.~2336--2360,
  \url{https://doi.org/10.1137/16M1094324}.

\bibitem{Satoh2002}
{\sc {Satoh, M.}}, {\em Conservative scheme for the compressible nonhydrostatic
  models with the horizontally explicit and vertically implicit time
  integration scheme}, Mon. Weather Rev., 130 (2002), pp.~1227--1245,
  \url{https://doi.org/10.1175/1520-0493(2002)130<1227:csftcn>2.0.co;2}.

\bibitem{SimmonsBurridge1981}
{\sc {Simmons, A.J.} and {Burridge, D.M.}}, {\em An energy and angular-momentum
  conserving vertical finite-difference scheme and hybrid vertical
  coordinates}, Mon. Weather Rev., 109 (1981), pp.~758--766,
  \url{https://doi.org/10.1175/1520-0493(1981)109<0758:AEAAMC>2.0.CO;2}.

\bibitem{tappwhite1976}
{\sc {Tapp, M.} and {White, P.}}, {\em A non-hydrostatic mesoscale model}, Q.
  J. Roy. Meteor. Soc., 102 (1976), pp.~277--296,
  \url{https://doi.org/10.1002/qj.49710243202}.

\bibitem{hommenh}
{\sc {Taylor, M.}, {Steyer, A.}, {Guba, O.}, {Hall, D.}, {Ullrich, P.},
  {Woodward, C.}, {Reynolds, D.}, {Gardner, D.}, and {Vogl, C.}}, {\em The
  {E3SM} non-hydrostatic atmosphere dynamical core}, preprint,  (2019).

\bibitem{ITT1997}
{\sc {Taylor, M.}, {Tribbia, J.}, and {Iskandarani, M.}}, {\em The spectral
  element method for the shallow water equations on the sphere}, J. Comput.
  Phys., 130 (1997), pp.~92--108,
  \url{https://doi.org/https://doi.org/10.1006/jcph.1996.5554}.

\bibitem{FournierTaylor2010}
{\sc {Taylor, M.A.} and {Fournier, A.}}, {\em A compatible and conservative
  spectral element method on unstructured grids}, J. Comput. Phys., 229 (2010),
  pp.~5879--5895, \url{https://doi.org/10.1016/j.jcp.2010.04.008}.

\bibitem{Thuburn2012}
{\sc {Thuburn, J.}}, {\em Basic dynamics relevant to the design of dynamical
  cores}, in Numerical Techniques for Global Atmospheric Models, P.~H.
  Lauritzen, C.~Jablonowski, M.~A. Taylor, and R.~D. Nair, eds., Springer,
  2012.

\bibitem{UllrichJablonowski2012}
{\sc {Ullrich, P.} and {Jablonowski, C.}}, {\em Operator-split
  {R}unge-{K}utta-{R}osenbrock methods for nonhydrostatic atmospheric models},
  Mon. Weather Rev., 140 (2012), pp.~1257--1284,
  \url{https://doi.org/10.1175/MWR-D-10-05073.1}.

\bibitem{dcmip2012}
{\sc {Ullrich, P.}, {Jablonowski, C.}, {Kent, J.}, {Lauritzen, P.}, {Nair, R.},
  and {Taylor, M.}}, {\em Dynamical core model intercomparison project
  ({DCMIP}) 2012 test case document v.1.7}, Tech. Rep.,  (2012).

\bibitem{ullrichetal2018}
{\sc {Ullrich, P.}, {Reynolds, D.}, {Guera, J.}, and {Taylor, M.}}, {\em Impact
  and importance of hyperdiffusion on the spectral element method: A linear
  dispersion analysis}, J. Comput. Phys., 375 (2018), pp.~427--446,
  \url{https://doi.org/10.1016/j.jcp.2018.06.035}.

\bibitem{Vallis2017}
{\sc G.~K. Vallis}, {\em Atmospheric and Oceanic Fluid Dynamics: Fundamentals
  and Large-Scale Circulation}, Cambridge University Press, 2017,
  \url{https://doi.org/10.1017/9781107588417}.

\bibitem{Houwen1972}
{\sc {van der Houwen, P.}}, {\em Explicit {R}unge-{K}utta formulas with
  increased stability boundaries}, Numer. Math., 20 (1972), pp.~149--164,
  \url{https://doi.org/10.1007/BF01404404}.

\bibitem{Houwen1977}
{\sc {van der Houwen, P.}}, {\em Construction of integration formulas for
  initial-value problems}, North-Holland, Amsterdam, 1977.

\bibitem{Houwen1996}
{\sc {van der Houwen, P.}}, {\em The development of {R}unge-{K}utta methods for
  partial differential equations}, Appl. Numer. Math., 20 (1996), pp.~261--272,
  \url{https://doi.org/10.1016/0168-9274(95)00109-3}.

\bibitem{Vichnevetsky1983}
{\sc {Vichnevetsky, R.}}, {\em New stability theorems concerning one-step
  methods for ordinary differential equations}, Math. Comput. Simulat., XXV
  (1983), pp.~199--205, \url{https://doi.org/10.1016/0378-4754(83)90092-7}.

\bibitem{voglpaper2019}
{\sc {Vogl, C.}, {Steyer, A.}, {Reynolds, D.}, {Ullrich, P.}, and {Woodward,
  C.}}, {\em {Evaluation of implicit-explicit Runge-Kutta integrators for the
  HOMME-NH dynamical core}}, Preprint,  (2019).

\bibitem{LWW2013}
{\sc {Weller, H.}, {Lock, S.-J.}, and {Wood, N.}}, {\em Runge-{K}utta {IMEX}
  schemes for the horizontally explicit/vertically implicit ({HEVI}) solution
  of wave equations}, J. Comput. Phys., 252 (2013), pp.~365--381,
  \url{https://doi.org/10.1016/j.jcp.2013.06.02}.

\bibitem{BSZ2016}
{\sc {Zhang, H.}, {Sandu, A.}, and {Blaise, S.}}, {\em High order
  implicit-explicit general linear methods with optimized stability regions},
  SIAM J. Sci. Comput., 38 (2016), pp.~A1430--A1453,
  \url{https://doi.org/10.1137/15M1018897}.

\end{thebibliography}

\end{document}